\documentclass[12pt,onecolumn,draftcls,journal,letterpaper]{IEEEtran}

\IEEEoverridecommandlockouts                              
\usepackage{amssymb,amsmath,color,psfrag}
\usepackage{graphicx}
\usepackage{algorithm,algorithmic,xspace}
\usepackage{epsfig}
\usepackage{verbatim}
\usepackage{subfig}
\usepackage{mathrsfs}
\usepackage{url}

\newcommand{\real}{{\mathbb{R}}}

\newcommand{\CC}{\mathcal{C}}

\newcommand{\GG}{\mathcal{G}}

\newcommand{\LL}{\mathcal{L}}
\newcommand{\MM}{\mathcal{M}}
\newcommand{\NN}{\mathcal{N}}
\newcommand{\PP}{\mathcal{P}}

\newcommand{\TT}{\mathcal{T}}
\newcommand{\UU}{\mathcal{U}}

\newcommand{\subj}{\text{subj. to}}

\newcommand{\qxi}{\phi} 
\newcommand{\xb}{\textbf{x}} 

\newcommand{\feqmfd}{\ell} 
\newcommand{\xeqmfd}{\eta} 

\newcommand{\xrm}{\textrm{x}} 
\newcommand{\urm}{\textrm{u}}

\newtheorem{theorem}{Theorem}[section]
\newtheorem{proposition}[theorem]{Proposition}

\newtheorem{lemma}[theorem]{Lemma}

\newtheorem{remark}[theorem]{Remark}

\newcommand\oprocendsymbol{\hbox{$\square$}}
\newcommand\oprocend{\relax\ifmmode\else\unskip\hfill\fi\oprocendsymbol}

\graphicspath{{figs_tcst/}}

\begin{document}

\title{Optimal control based dynamics exploration\\ of a rigid car with
  longitudinal load transfer\thanks{An early short version of this work appeared
    as~\cite{AR-GN-JH:10}: the current article includes a much improved
    comprehensive treatment, new results on the proposed model, revised complete
    proofs for all statements, and a new experimental computation scenario.}}

\author{Alessandro Rucco, Giuseppe Notarstefano, and John Hauser
    \thanks{
    A. Rucco and G. Notarstefano are with
    the Department of Engineering, University of Lecce (Universit\`a del Salento), 
    Via per Monteroni, 73100 Lecce, Italy, \texttt{\{alessandro.rucco, giuseppe.notarstefano\}@unisalento.it}}
    \thanks{ J. Hauser is with the Department of Electrical and
    Computer Engineering, University of Colorado, Boulder, CO 80309-0425, USA,
    \texttt{hauser@colorado.edu} }
    }

\maketitle

\begin{abstract}
  In this paper we provide optimal control based strategies to explore the dynamic
  capabilities of a single-track car model which includes tire models and
  longitudinal load transfer.
  Using an explicit formulation of the holonomic constraints imposed on the
  unconstrained rigid car, we design a car model which includes load transfer
  without adding suspension models.
  With this model in hand, we perform an analysis of the equilibrium manifold of
  the vehicle. That is, we design a continuation and predictor-corrector
  numerical strategy to compute cornering equilibria on the entire range of
  operation of the tires.
  Finally, as main contribution of the paper, we explore the system dynamics by
  use of nonlinear optimal control techniques. Specifically, we propose a
  combined optimal control and continuation strategy to compute aggressive car
  trajectories and study how the vehicle behaves depending on its parameters.
  To show the effectiveness of the proposed strategy, we compute aggressive
  maneuvers of the vehicle inspired to testing maneuvers from virtual and real
  prototyping.
\end{abstract}



\section{Introduction}

A new emerging concept in vehicle design and development is the use of
\emph{virtual vehicles}, i.e., software tools that reproduce the behavior of the
real vehicle with high fidelity \cite{TD:95}, \cite{RF-AB:06}. They allow car
designers to perform dynamic tests before developing the real prototype, thus
reducing costs and time to market. This engineering area is called \emph{virtual
  prototyping}.

%
In order to explore the dynamic capabilities of a car vehicle or to design
control strategies to drive it, it is important to develop dynamic models that
capture interesting dynamic behaviors and, at the same time, can be described by
ordinary differential equations of reasonable complexity.
Many models have been introduced in the literature to describe the motion of a
car vehicle both for simulation and control \cite{WFI:95,TG:92,JYW:01,UK-LN:05,RR:06}.
%
The \emph{bicycle model} is a planar rigid model that approximates the vehicle as a
rigid body with two wheels. It is widely used in the literature
since it captures many interesting phenomena concisely.  However, this model
does not capture some important dynamic effects. One of them is load
transfer. The most natural way to model load transfer would be to add suspension
models. 
Using an idea independently developed in \cite{EV-PT-JL:08}, see also
\cite{EV-EF-PT:10}, we will model tire normal loads by means of the reaction
forces generated at the vehicle contact points by the ground. This allows us to
model load transfer without adding suspension models, thus with a reasonable
increase in the model complexity.

Car dynamics analysis at maximum performance has been widely investigated in the
literature. We provide an overview of the relevant literature for our
work. First, an analysis of the equilibrium manifold for race vehicles is
performed in \cite{EO-SH-HT-SD:98} and \cite{EV-EF-PT:10}. In particular,
existence and stability of ``cornering equilibria'', i.e. steady-state
aggressive turning maneuvers, and bifurcation phenomena are investigated. 
In \cite{MA:06} the physical parameters affecting (drifting) steady-state
cornering maneuvers are examined both in simulation and experiments.
 Aggressive non-steady state cornering maneuvers for rally vehicles were
 proposed in \cite{EV-PT-JL:08} (see also \cite{EF:08}), and
 \cite{YI-LI-LU-LIU:11}. In \cite{EV-PT-JL:08} trajectories comparable with real
 testing driver maneuvers were obtained by solving a suitable minimum-time
 optimal control problem, whereas in \cite{YI-LI-LU-LIU:11} stability and
 agility of these maneuvers were studied.  In \cite{DC-RS-PS:00a} and
 \cite{FK-JF-CK-SS:11} minimum-time trajectories of formula one cars were
 designed by means of numerical techniques based on Sequential Quadratic
 Programming and Direct Multiple Shooting, respectively.
In \cite{DC-RS-PS:00b,DC-RS-PS:02} the influence of the vehicle mass and center
of mass on minimum-time trajectories was studied. Recently, \cite{And-Iagn:10},
a constrained optimal control approach was pursued for optimal trajectory
planning in a constrained corridor. A Model Predictive Control approach is used
to control the vehicle along the planned trajectory. Model Predictive Control
for car vehicles has been widely investigated, see, e.g.,
\cite{PF-FB-JA-HET-DH:07a,PF-FB-JA-HET-DH:07b}.  It is worth noting that the
optimal control strategy proposed in the paper for trajectory exploration can be
also used in a Model Predictive Control scheme to track a desired curve.

The contributions of the paper are as follows.
First, we develop a single-track model of rigid car that extends the
capabilities of the well known bicycle model and generalizes the one introduced
in \cite{EV-PT-JL:08}. We call this model LT-CAR, where ``LT'' stands for load
transfer. Our LT-CAR model differs from the one in \cite{EV-PT-JL:08} for an
additional term in the normal forces that depends on the square of the
yaw-rate. As an ``educational'' contribution, we provide a rigorous
derivation of the proposed model by use of a Lagrangian approach.
%
This novel derivation can be extended to a wide class of
mechanical systems subject to a special set of external forces, whose dependence
on internal variables can be modeled by suitable reaction forces (as, e.g.,
motorcycles \cite{PM-JH:09,AS-JH-AB:12}).
%
%
%

%

Second, with this model in hand, we perform an analysis of the equilibrium
manifold of the vehicle. Namely, we study the set of cornering equilibria,
i.e. trajectories of the system that can be performed by use of constant inputs.
%
%
We design a numerical strategy based on zero finding techniques combined with
predictor-corrector continuation methods \cite{ELA-KG:90} to compute the
equilibrium manifold \emph{on the entire range of operation of the tires}.
%
At the best of our knowledge this is the first strategy to systematically
explore the equilibrium manifold on the entire tire range. For example, in
\cite{EO-SH-HT-SD:98,EV-EF-PT:10,AS-JH-AB:12}, only some
snapshots of the equilibrium manifold are computed and analyzed.
To show the effectiveness of the proposed method, we show slices of the
equilibrium manifold using the parameters of a sports car with rear-wheel drive
transmissions given in \cite{GG:00}. Moreover, we investigate the structure of
the equilibrium manifold with respect to variations in the horizontal position
of the center of mass. Moving the center of mass from the rear to the front
causes a significant change in the structure of the equilibrium manifold giving
rise to interesting bifurcations.

Third and final, we develop a trajectory exploration strategy, based on
nonlinear optimal control techniques introduced in
\cite{JH:02}, 
to explore aggressive vehicle trajectories at the limits of its dynamic
capabilities.
%
Clearly, given a vehicle model one could just pose a nonlinear optimal control
problem and apply standard machinery to solve it. The strategy that we propose
goes beyond this straightforward machinery. Indeed, optimal control problems are
infinite dimensional optimization problems that, therefore, can lead to local
minima with significantly different structures. This is crucial in vehicle
dynamics exploration, because the local minimum could be a trajectory that is
not representative of the actual vehicle behavior.

The main idea of the proposed strategy is the following. Given a desired
path-velocity profile, we design a full (state-input) desired curve and look for
a vehicle trajectory minimizing a weighted $L_2$ distance from the desired
curve. In order to solve this optimal control problem, we design an initial
``nonaggressive'' desired curve and morph it to the actual one. For each
temporary desired curve, we solve the optimal control problem by initializing
the numerical method with the optimal trajectory at the previous step. This
continuation idea resembles the learning process of a test-driver when testing
the capabilities of a real vehicle.

%

We show the strategy effectiveness in understanding complex car trajectories on two testing maneuvers.
%
%
In the first test, we perform an aggressive maneuver by using a multi-body
software, Adams, to generate the desired curve. The objective of this choice is
twofold: (i) we show the effectiveness of the exploration strategy in finding an
LT-CAR trajectory close to the desired curve, and (ii) we validate the LT-CAR
model by showing that the desired curve, which is a trajectory of the full Adams
model, is in fact ``almost'' a trajectory of the LT-CAR model. In the second
test, we perform a constant speed maneuver on a real testing track (a
typical maneuver for real vehicle testing). We show how to design a full
state-input desired curve (from the assigned path and speed) by use of a
quasi-static approximation and compute an optimal trajectory that shows a
typical driver behavior in shaping the path to keep the speed constant.

The rest of the paper is organized as follows. In Section~\ref{sec:model} we
introduce and develop the LT-CAR model. In Section~\ref{sec:eq_manif} we
characterize the equilibrium manifold 
and provide a comparison with the standard bicycle model. Finally, in
Section~\ref{sec:traj_explor} we describe the strategy for trajectory
exploration and provide numerical computations performed on virtual and real
testing tracks.




\section{LT-CAR model development}
\label{sec:model}
In this section we introduce the car model with load transfer (LT-CAR) studied
in the paper. This model is an extension of the one proposed in
\cite{EV-PT-JL:08,EV-EF-PT:10}.
%
We model the car as a single planar rigid body with five degrees of freedom
(three displacements and two rotations) and then constrain it to move in a plane
(three degrees of freedom) interacting with the road at two body-fixed contact
points. The center of mass and the two contact points all lie within a plane
with the center of mass located at distance $b$ from the rear contact point and
$a$ from the front one, respectively. Each contact-point/road-plane interaction
is modeled using a suitable tire model as, e.g., the \emph{Pacejka} model
\cite{HBP:02}. A planar view of the rigid car model is shown in
Figure~\ref{fig:rigid_car}.

\begin{figure}[htpb]
\vspace{-2ex}
  \begin{center}
    \includegraphics[scale=.18]{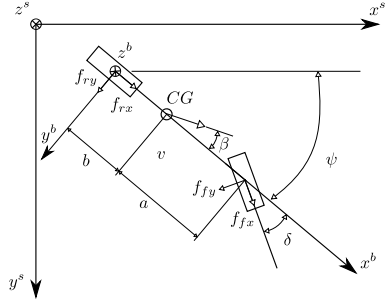}
    \caption{\small LT-CAR model. The figure show the quantities used to describe the
      model.}
    \label{fig:rigid_car}
  \end{center}
\vspace{-4ex}
\end{figure}

The body-frame of the car is attached at the rear contact point with $x$-$y$-$z$
axes oriented in a forward-right-down fashion. We let $\xb=[x, y,
z]^T\in\real^3$ and $R\in SO(3)$ denote the position and orientation of the
frame with respect to a fixed spatial-frame with $x$-$y$-$z$ axes oriented in a
north-east-down fashion. $R$ maps vectors in the body frame to vectors in the
spatial frame so that, for instance, the spatial angular velocity $\omega^s$ and
the body angular velocity $\omega^b$ are related by $\omega^s = R\omega^b$ and
$\omega^b=R^T\omega^s$. Similarly, $\xb^s = \xb + R\,\xb^b$ gives the spatial
coordinates of a point on the body with body coordinates $\xb^b \in \real^3$.
The orientation $R$ of the (unconstrained) rigid car model can be parameterized
(using Roll-Pitch-Yaw parametrization) as follows
\[
\small{
  R=R(\psi, \theta)=R_z(\psi)R_y(\theta) =
  \left[
    \begin{array}{ccc}
      c_{\psi}c_{\theta} & -s_{\psi} & c_{\psi}s_{\theta} \\
      s_{\psi}c_{\theta} & c_{\psi} & s_{\psi}s_{\theta} \\
      -s_{\theta} & 0 & c_{\theta} \\
    \end{array}
  \right],
}
\]
where $\theta$ and $\psi$ are respectively the pitch and yaw angles (we use the
notation $c_{\psi}=\cos(\psi)$, etc.).
In the rest of the paper, for brevity, we use the notation
$\qxi = [\psi,\theta]^T$.
%
The vector
\[
  q = [x,y,\psi,z,\theta]^T = [q_r,q_c]^T
\]
provides a valid set of generalized coordinates for dynamics calculations.
The coordinates
$q_r = [x, y, \psi]^T$ are the \emph{reduced} unconstrained coordinates,
while
$q_c = [z,\theta]^T$ are the \emph{constrained} ones. 
%

\subsection{Tire models} 
\label{sec:tires_gen_forces}
We model the tire forces by using a suitable version of the Pacejka's Magic
Formula \cite{HBP:02}. 
Before, we clarify our notation. We use a subscript ``$f$'' (``$r$'') for
quantities of the front (rear) tire. When we want to give a generic expression
that holds both for the front and the rear tire we just suppress the
subscript. Thus, for example, we denote the generic normal tire force $f_{z}$,
meaning that we are referring to $f_{fz}$ for the front tire and $f_{rz}$ for
the rear one.

The rear and front forces tangent to the road plane, $f_{x}$ and $f_{y}$, depend
on the normal force and on the longitudinal and lateral slips.
The longitudinal slip $\kappa$ is the normalized difference between the angular
velocity of the driven wheel $\omega_w$ and the angular velocity of the
free-rolling $\omega_0 = v_{cx}/r_w$, with $v_{cx}$ the contact point
longitudinal velocity,
\[
\kappa = \frac{\omega_w - \omega_0}{\omega_0} = -\frac{v_{cx} - r_w
  \omega_w}{v_{cx}}.
\]
The lateral slip (or sideslip) $\beta$ is defined as
$\tan{\beta} = v_{cy} / v_{cx}$, 
%
with $v_{cy}$ the lateral velocity.  We assume that the rear and front forces
tangent to the road plane, $f_{x}$ and $f_{y}$, depend linearly on the normal
forces.
%
Thus, the combined slip forces are
\begin{equation*}
  \begin{split}
    f_x=-f_zf_{x0}(\kappa)g_{x\beta}(k,\beta)=-f_z\mu_x(\kappa,\beta)\\
    f_y=-f_zf_{y0}(\beta)g_{yk}(k,\beta)=-f_z\mu_y(\kappa,\beta),
  \end{split}
\end{equation*}
where the pure longitudinal slip $f_{x0}(\kappa)$, the pure lateral slip
$f_{y0}(\beta)$ and the loss functions for combined slip
$g_{x\beta}(\kappa,\beta)$ and $g_{yk}(\kappa,\beta)$ are defined in
Appendix~\ref{APP:params} together with the values of the parameters used in the paper.


The front forces expressed in the body frame, $f_{fx}^b$ and
$f_{fy}^b$, are obtained by rotating the forces in the tire frame according to
the steer angle $\delta$, so that, e.g.,
$f_{fx}^b = f_{fx} c_\delta - f_{fy} s_\delta$.
Substituting the above expressions for $f_{fx}$ and $f_{fy}$, we get
\[
\begin{split}
  f_{fx}^b &= -f_{fz}\big(\mu_{fx}(\kappa_f,\beta_f) c_\delta - \mu_{fy}(\kappa_f,\beta_f)s_\delta\big)
  := -f_{fz} \tilde{\mu}_{fx}(\kappa_f,\beta_f,\delta).
\end{split}
\]
In the rest of the paper, abusing notation, we will suppress the `tilde' and use
$\mu_{fx}(\kappa_f,\beta_f,\delta)$ to denote
$\tilde{\mu}_{fx}(\kappa_f,\beta_f,\delta)$.

We assume to control the longitudinal slips $\kappa_r$ and $\kappa_f$. We want
to point out that, depending on the analysis one can control the two slips
independently or a combination of the two. For example, in the equilibrium
manifold analysis and in the second trajectory exploration scenario we will set
$\kappa_f=0$ and use only $\kappa_r$ as control input (rear-wheel drive).
%
%
 %
Thus, the \emph{control inputs} of
the car turn to be:
\begin{itemize}
\item $\kappa_r$ and $\kappa_f$, the rear and front longitudinal slips, and
\item $\delta$, the front wheel steer angle.
\end{itemize}

\begin{remark}[Longitudinal slip as control input]
  The use of the longitudinal slip as control input is present in the
  literature, e.g., \cite{PF-FB-JA-HET-DH:07b} and \cite{PM-JH:09}. This choice
  does not limit the applicability of our analysis. Indeed, wheel torques can be
  easily computed once a trajectory is computed. \oprocend
\end{remark}


Next, we introduce a simplified tire model that will be used to design
approximate trajectories (trajectories of a simplified car model) characterized
by contact forces that can not be generated by the Pacejka's model.
This simplified tire model, \cite{WFI:95,GG:00,RF-AB-GN:05}, relies on
the following assumptions: (i) the longitudinal force is directly controlled,
(ii) the relationship between the lateral force $f_y$ and the sideslip $\beta$
is linear, and (iii) the longitudinal and lateral forces, $f_x$ and $f_y$, are
decoupled.  We call the simplified car model obtained by using this tire
approximation the Linear Tire LT-CAR (LT$^2$-CAR).
%
Figure~\ref{fig:pacejka} shows the plots of the longitudinal and lateral forces
$f_x$ and $f_y$ for the Pacejka's and linear tire models.


\begin{figure}[htbp] %
\begin{center}
\subfloat[]{\includegraphics[scale=.27]{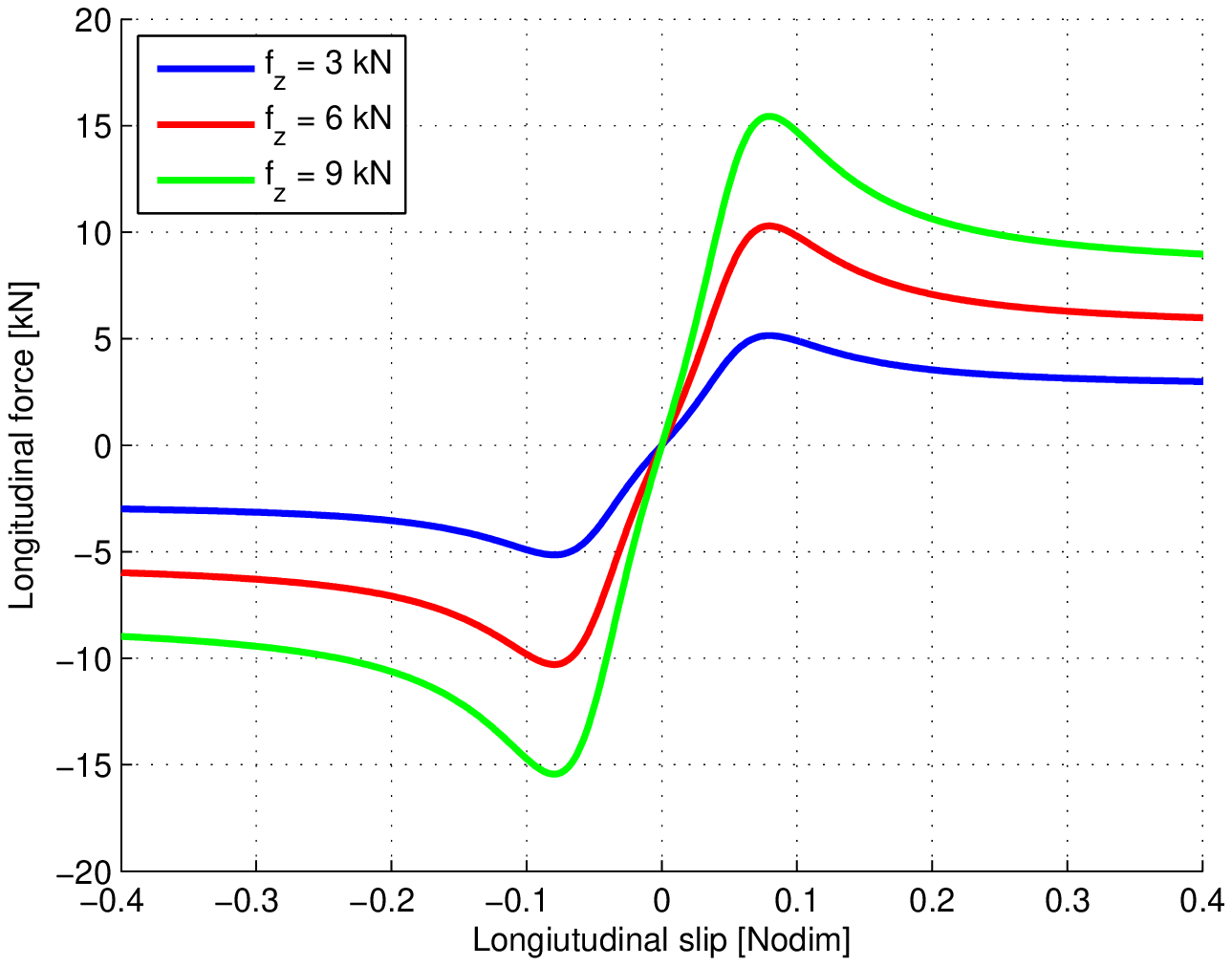} \label{fig:pacejka_long}}
\subfloat[]{\includegraphics[scale=.27]{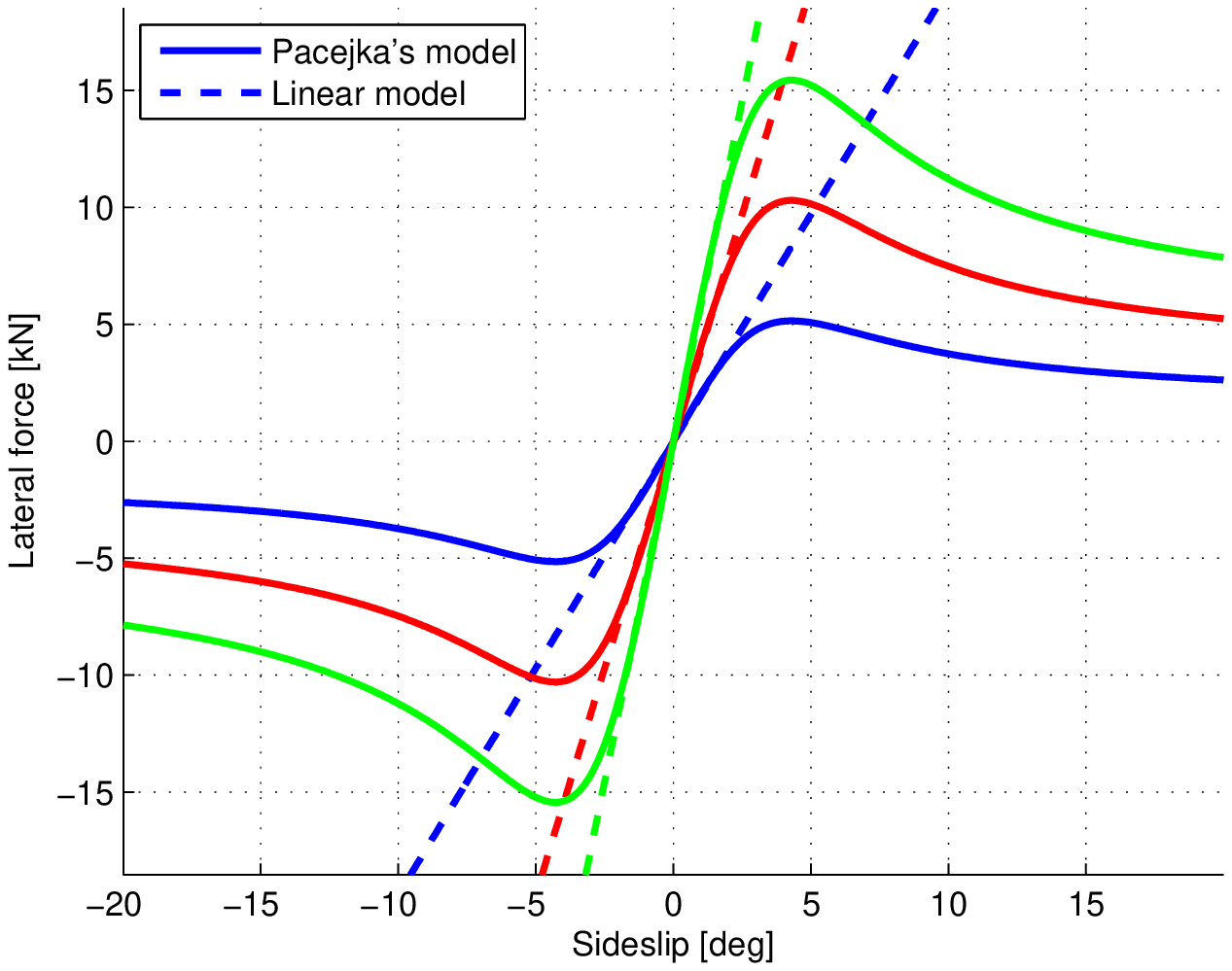} \label{fig:pacejka_lat}}
\subfloat[]{\includegraphics[scale=.27]{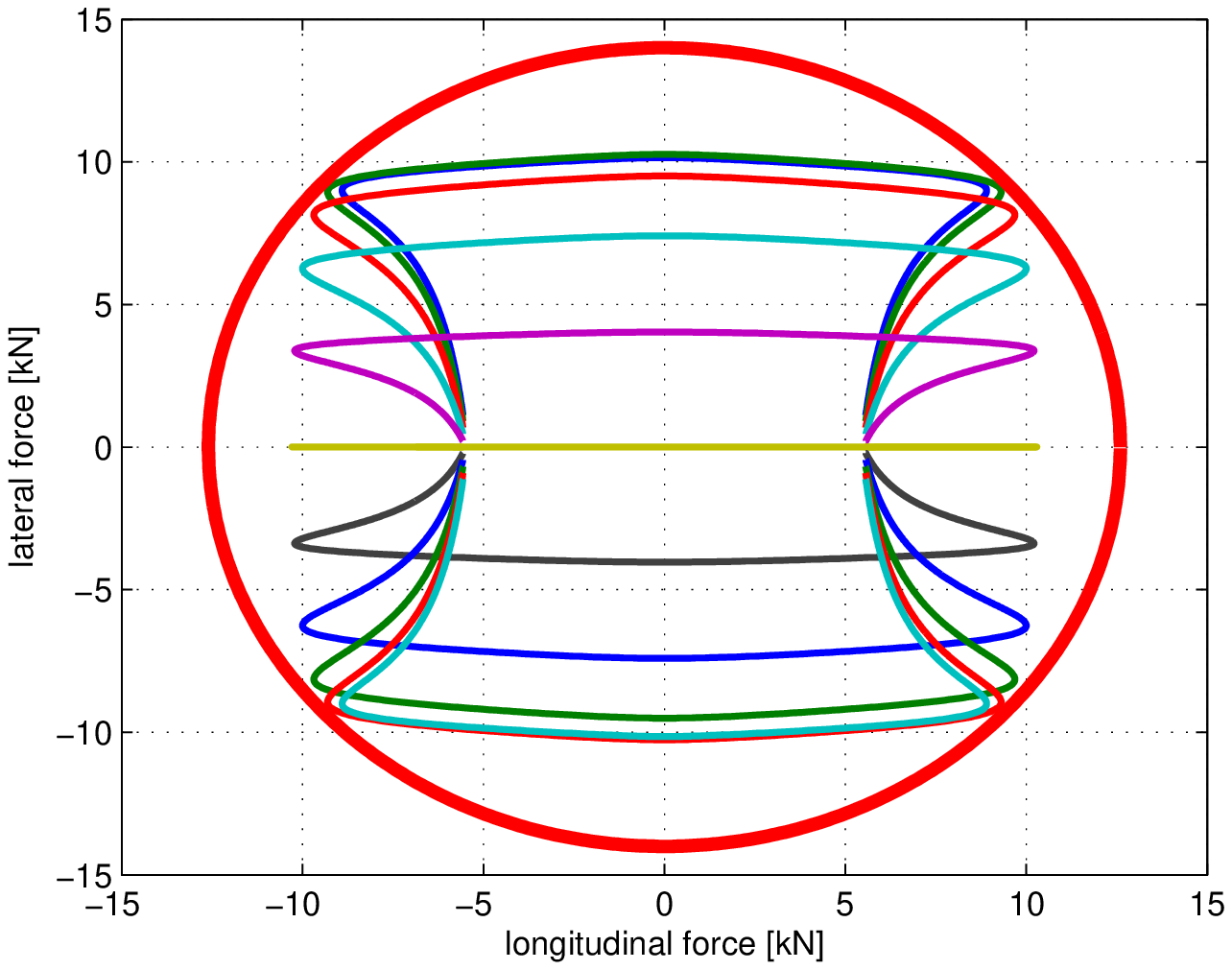} \label{fig:pacejka_ell}}
\caption{\small Pure longitudinal (a) and
    lateral (b) forces are plotted as function of respectively longitudinal and
    lateral slip for three values of the normal force. In (b) the simplified
    tire model (dashed line) is also shown. The longitudinal versus lateral
    force is plotted as function of the longitudinal slip for different values
    of the sideslip (c). The ellipse of maximum tire forces is shown in solid
    red.}\label{fig:pacejka}
\end{center}
\end{figure}

\vspace{-4ex}

\begin{remark}[Other tire models]
  Tires are one of the key components of the vehicle and have an important
  impact on the performance. To capture the complex behavior of the tires
  several models have been developed in the literature \cite{HBP:02},
  \cite{CC:95}, \cite{CC:03}. We highlight that the LT-CAR model can be
  developed with any tire model (not necessarily the Pacejka's one). \oprocend
\end{remark}

\subsection{Constrained Lagrangian dynamics}
\label{subsec:constr_model}
Next, we develop the constrained planar model of the rigid car and include load
transfer.
To describe the motion in the plane, we derive the equations of motion of the
unconstrained system and explicitly incorporate the constraints (rather than
choosing a subset of generalized coordinates). This allows us to have an
explicit expression for the normal (constraint) forces.

We derive the dynamics of the unconstrained system via the Euler-Lagrange
equations. To do this, we define the Lagrangian $\mathcal{L}$ as the difference
between the kinetic and potential energies
$\mathcal{L}(q,\dot{q})=T(q,\dot{q})-V(q)$.
The equations of motion for the unconstrained system are given
by the Euler-Lagrange equations
\begin{equation}
  \label{eq:lagrange}
  \frac{d}{dt}\frac{\partial \mathcal{L}}{\partial\dot{q}}^T - \frac{\partial \mathcal{L}}{\partial q}^T = U
\end{equation}
where $U$ is the set of generalized forces.  Exploiting the Euler-Lagrange
equations, we get
\begin{equation}\label{eq:xypsiztheta}
  M(q)\ddot{q} + C(q,\dot{q})+G(q)=U
\end{equation}
 with $M(q)$, $C(q,\dot{q})$ and $G(q)$ respectively the
  mass matrix, and the Coriolis and gravity vectors.

  The longitudinal and lateral forces arising from the tire-road interactions at
  the front and rear contact points, $f = [f_{fx},f_{fy},f_{rx},f_{ry}]^T$, are
 converted into the generalized forces $U$ by using the \emph{principle of
    virtual work},
$<\!f,v_{cp}^b\!> = <U, \dot{q}>$, 
where $v_{cp}^b = [v_{fx}^b, v_{fy}^b, v_{rx}^b, v_{ry}^b]^T$ are the
longitudinal and lateral velocities at the front and rear contact
points. Computing the Jacobian matrix $J_f(\phi)$ mapping $\dot{q}$ to the front
and rear contact point velocities expressed in the body frame, $v_{cp}^b=
J_f(\phi)\dot{q}$, we get
$<J_f(\phi)^T f, \dot{q}> = <U, \dot{q}>$,
so that
\[
U = J_f^T(\phi) f.
\]
The front and rear contact points coordinates expressed in the body frame are
$\xb_f^b=[a+b,0,0]^T$ and $\xb_r^b=[0,0,0]^T$.  The coordinates in the spatial
frame, respectively $\xb_r^s = [x_r^s, y_r^s, z_r^s]^T$ and $\xb_f^s = [x_f^s,
y_f^s, z_f^s]^T$, are $\xb_r^s=\xb$ and $\xb_f^s=\xb+R\,\xb_f^b$, so that the
velocities in the spatial frame are 
\[
v_r^s = \dot{\xb} \qquad \text{and} \qquad
  v_f^s = \dot{\xb}+R\,\omega^b \times
  \xb_f^b
  = \dot{\xb}-R\,\xb_f^b \times \omega^b
  = \dot{\xb}-R\,\hat{\xb}_f^bJ_{\omega^b}(\qxi)\dot{\qxi},
\]
where $\hat{\xb}_f^b$ is the skew-symmetric matrix associated to the vector $\xb_f^b$,
while the velocities expressed in the body frame are
\[
  v_r^b = R^T \dot{\xb} = J_{v_r^b}(\qxi) \dot{q}\, \qquad \text{and} \qquad
  v_f^b = R^T \dot{\xb}-\hat{\xb}_f^bJ_{\omega^b}(\qxi)\dot{\qxi}
  = J_{v_f^b}(\qxi) \dot{q}\,.
\]
Thus, the Jacobian $J_f$ turns to be
\begin{align*}
  J_f(\qxi) &=
  \small{ \left[
      \begin{array}{c}
        J_{v_{fx}^b} \\
        J_{v_{fy}^b} \\
        J_{v_{rx}^b} \\
        J_{v_{ry}^b}
      \end{array}
    \right] } = \small{ \left[
      \begin{array}{ccccc}
        c_{\psi}c_{\theta} & s_{\psi}c_{\theta} & 0 & -s_{\theta} & 0 \\
        -s_{\psi} & c_{\psi} & (a+b)c_\theta & 0 & 0 \\
        c_{\psi}c_{\theta} & s_{\psi}c_{\theta} & 0 & -s_{\theta} & 0\\
        -s_{\psi} & c_{\psi} & 0 & 0 & 0\\
      \end{array}
    \right]} \,.
\end{align*}

Next, we constrain the contact points to the road plane in order to compute the
normal tire forces as reaction forces.
We impose the constraint that the rear and front contact points have zero
velocity along the $z$ axis. The velocity constraints are given by
$\dot{z}_r = e_3^TR^T\dot{\xb} = J_{v^b_{rz}}(\qxi)\dot{q} = 0$, and
$\dot{z}_f = e_3^T(R^T \dot{\xb}-\hat{\xb}_f^bJ_{\omega^b}(\qxi)\dot{\qxi}) = J_{v^b_{fz}}(\qxi)\dot{q} = 0$, where $e_3=[0, 0, 1]^T$, and $z_r$ ($z_f$) is the position of the rear (front) contact point
expressed in the body frame.
The front and rear constraints may be written in the form $A(q)\dot{q}=0$, where
\begin{equation}\label{eq:constr_matrix}
  \begin{split}
    A(q)  \!=\!
    \begin{bmatrix}
      J_{v^b_{fz}}(\qxi) \\
      J_{v^b_{rz}}(\qxi) \\
    \end{bmatrix}
    =\!
    \begin{bmatrix}
      c_{\psi}s_{\theta} & s_{\psi}s_{\theta} & 0 &  c_{\theta} & -(a+b) \\
      c_{\psi}s_{\theta} & s_{\psi}s_{\theta} & 0 & c_{\theta} & 0 \\
    \end{bmatrix}\!.
  \end{split}
\end{equation}
From the principle of virtual work, 
we get the vector of constraint generalized forces, $U_c$, in terms of the front
and rear normal contact point forces, $\lambda = [-f_{fz},~-f_{rz}]^T \in
\real^2$, as $U_c = -A^T(q)\lambda$.


In the next proposition we show that, under the linear dependence of the contact
point forces on the normal ones, the constrained system can be explicitly
written as an unconstrained ordinary differential equation.
Since the proof follows classical arguments from mechanics, we reported it in
Appendix~\ref{APP:proof_model} as a tutorial contribution.

\begin{proposition} 
\label{prop:car_model}
  Given the unconstrained car model with structure as in \eqref{eq:xypsiztheta}
  and constraints in \eqref{eq:constr_matrix}, the following holds true:
  \begin{enumerate}
  \item the dynamics of the constrained system can be written in terms of the
    unconstrained coordinates $q_r=[x,y,\psi]^T$ and the reaction forces $\lambda
    = [-f_{fz},~-f_{rz}]^T$ as
    \begin{equation}
        \tilde{\MM}(q_r)
        \left[
          \begin{array}{c}
            \ddot{q}_r \\
            \lambda \\
          \end{array}
        \right]
        +\CC(q_r,\dot{q}_r)+\GG(q_r)= \UU
        ,  
      \label{th:lagr_const_matrix}
    \end{equation}
    where
    \begin{align*}
      &\tilde{\MM}(q_r) = \left[
        \begin{array}{cc}
          \MM_{11}(q_r) & 0 \\
          \MM_{21}(q_r) & \MM_{22}(q_r) \\
        \end{array}
      \right]
      = \small{ \left[
          \begin{array}{ccc|cc}
            m & 0 & -mbs_{\psi} & 0 & 0 \\
            0 & m & mbc_{\psi} & 0 & 0 \\
            -mbs_{\psi} & mbc_{\psi} & I_{zz}+mb^2 & 0 & 0 \\
            \hline
            0 & 0 & 0 & 1 & 1 \\
            -mhc_{\psi} & -mhs_{\psi} & 0 & -(a+b) & 0
          \end{array}
        \right]} ,
    \end{align*}
    \begin{equation}
      \CC(q_r,\dot{q}_r)=
      \left[
        \begin{array}{c}
          \CC_{1}(q_r,\dot{q}_r) \\
          \CC_{2}(q_r,\dot{q}_r) \\
        \end{array}
      \right]=
      \small{
        \left[
          \begin{array}{c}
            -mbc_{\psi}\dot{\psi}^2 \\
            -mbs_{\psi}\dot{\psi}^2 \\
            0 \\
            \hline
            0 \\
            (I_{xz}+mhb)\dot{\psi}^2
          \end{array}
        \right]}, \;
      \GG(q_r) =
      \left[
        \begin{array}{c}
          \GG_{1}(q_r) \\
          \GG_{2}(q_r) \\
        \end{array}
      \right]=
      \small{
        \left[
          \begin{array}{c}
            0 \\
            0 \\
            0 \\
            \hline
            -mg \\
            mgb
          \end{array}
        \right]},
      \label{th:lagr_const_G}
    \end{equation}
    \begin{equation}
      \small{
        \UU =
        \left[
          \begin{array}{c}
            \UU_1 \\ \hline
            0
          \end{array}
        \right] =
        \left[
          \begin{array}{cccccc}
            c_\psi   & -s_\psi   & c_\psi   &  -s_\psi \\
            s_\psi   & c_\psi   & s_\psi   &   c_\psi \\
            0  &        a+b  &       0  &        0  \\ \hline
            0  &        0  &       0  &        0  \\
            0  &        0  &       0  &        0  \\
          \end{array}
        \right]
        \left[
          \begin{array}{c}
            f_{fx}\\
            f_{fy}\\
            f_{rx}\\
            f_{ry}
          \end{array}
        \right];}
      \label{th:lagr_const_U}
    \end{equation}
  \item the subsystem
    \begin{equation} \label{eq:xypsi}
      \MM_{11}(q_r)\ddot{q}_r+\CC_1(q_r,\dot{q}_r)+\GG_1(q_r)=\UU_1
    \end{equation}
    is a Lagrangian system obtained from a suitable \emph{reduced Lagrangian}
    $\LL_r(q_r)$, with constraint forces $\lambda$ determined by
    \[
    \MM_{21}(q_r)\ddot{q}_r+\MM_{22}(q_r)\lambda+\CC_2(q_r,\dot{q}_r)+\GG_2(q_r)=0;
    \]
  \item under the assumption that the forces $f$ depend linearly on the reaction forces, i.e. $f = F\lambda$, the car dynamics turn to be
    \begin{equation}
         \MM(q_r, \mu)
        \left[
          \begin{array}{c}
            \ddot{q}_r \\
            \lambda \\
          \end{array}
        \right]
        +\CC(q_r,\dot{q}_r)+\GG(q_r)=
        0
       \label{th:dynamic_eq}
    \end{equation}
    with
    \begin{align*}
      &\MM(q_r,\mu) = \left[
        \begin{array}{cc}
          \MM_{11}(q_r) & \MM_{12}(q_r,\mu) \\
          \MM_{21}(q_r) & \MM_{22}(q_r) \\
        \end{array}
      \right].
    \end{align*}
  \end{enumerate}
\end{proposition}

\begin{remark}
 Equation \eqref{th:dynamic_eq} can be exploited as
  \[
  \begin{split}
    \ddot{q}_r &=\small{-(\MM_{11}+\MM_{12}\MM_{22}^{-1}\MM_{21})^{-1}}  
[\CC_1+\GG_1+\MM_{12}\MM_{22}^{-1}(\CC_2+\GG_2)]\\
    \lambda &= -\MM^{-1}_{22}(\CC_2+\GG_2+\MM_{21}\ddot{q}_r) .
  \end{split}
  \]
  From this expression it is clear that we have a dynamic model explicitly
  depending on the unconstrained coordinates $x$, $y$ and $\psi$ and an explicit
  expression for the reaction forces that can be used to calculate the normal
  loads for the tire forces.  \oprocend
\end{remark}

\subsection{Model well-posedness and load transfer analysis}
An important aspect to investigate is the well-posedness of the constrained
model \eqref{th:dynamic_eq}.  Differently from the standard unconstrained
equations of motion, as in \eqref{eq:xypsiztheta}, for which the mass matrix is
always positive definite (and thus invertible), the invertibility of the matrix
$\MM(q_r, \mu)$ depends on the model and tire parameters.
\begin{proposition}
\label{prop:RoV}
  The LT-CAR model is well-posed if the following
  inequalities are satisfied
  \begin{equation}
  \label{eq:RoV}
\begin{split}
\mu_{rx}<\frac{I_{xz}\dot{\psi}^2}{mgh}+\frac{b}{h} \;\; \text{and} \;\;
\mu_{fx}>\frac{I_{xz}\dot{\psi}^2}{mgh}-\frac{a}{h}.
\end{split}
\end{equation}
\end{proposition}
\begin{proof}
  By means of simple operations on the $\MM$ matrix, we can 
  compute the normal forces
    \begin{equation} \label{eq:load_transfer}
      f_{fz} = \frac{mgb - mgh \mu_{rx} + I_{xz}\dot{\psi}^2}{h\left(\mu_{rx}-\mu_{fx}\right)-(a+b)} \, , \qquad
      f_{rz} = \frac{mga + mgh \mu_{fx} - I_{xz}\dot{\psi}^2}{h(\mu_{rx}-\mu_{fx})-(a+b)} \, .
  \end{equation}
  In order for the model to be valid, both the two reaction forces,
  $f_{fz}$ and $f_{rz}$, need to be negative. Indeed, the ground is a unilateral
  constraint and, therefore, cannot generate a positive reaction force. Clearly,
  if the two conditions in \eqref{eq:RoV} are satisfied, the denominator of the
  two reaction forces is negative and both the two nominators are positive, thus
  concluding the proof.
\end{proof}


\begin{remark}
  From the combined slip Pacejka's formulas, $\mu_{rx}$ and
  $\mu_{fx}$ are bounded by
  \begin{align*}
    |\mu_{rx}| &= |f_{rx0}(\kappa_r)g_{rx\beta}(\kappa_r,\beta_r)| \leq d_x^r \\
    |\mu_{fx}| &= |c_{\delta} f_{fx0}(\kappa_f)g_{fx\beta}(\kappa_f,\beta_f)-s_{\delta}f_{fy0}(\beta_f)g_{fx\beta}(\kappa_f, \beta_f)| \leq (d_x^f+d_y^f).
  \end{align*}
  Thus, for the data provided in Appendix~\ref{APP:params}, the conditions of
  Proposition~\ref{prop:RoV} are always satisfied for ``reasonable'' values of
  $\dot{\psi}$ (e.g., for $|\dot\psi|<2\pi$ rad/s). \oprocend
\end{remark}

Equations \eqref{eq:load_transfer} 
show the influence of the front and rear longitudinal force coefficients on the load transfer.
A sharp acceleration, due to a high (positive) value of $\mu_{rx}$, increases
the load on the rear wheel while reducing the load on the front.
Similarly, a hard braking, mainly due to a high (negative) value of $\mu_{fx}$,
increases the front normal load while reducing the normal load on the rear.
%
In particular, if $\mu_{rx}=\frac{I_{xz}\dot{\psi}^2+mgb}{mgh}$, then the front
wheel leaves the ground, thereby producing a ``wheelie''; if
$\mu_{fx}=\frac{I_{xz}\dot{\psi}^2-mga}{mgh}$, then the rear wheel leaves the
ground, thereby producing a ``stoppie''.  In Figure~\ref{fig:RoV}, we provide a
graphical representation of the model validity region.
%
\begin{figure}[htbp]
  \begin{center}
    \includegraphics[scale=.4]{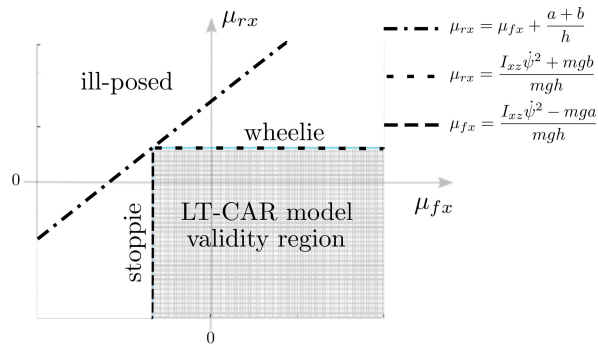}
    \caption{\small Well-posedness interpretation of the
      model. LT-CAR model validity region: the wheelie and stoppie
      conditions are avoided, i.e. $-f_{fz}>0$ and $-f_{rz}>0$.}
    \label{fig:RoV}
  \end{center}
\end{figure}

\subsection{Dynamics in the body frame}
\label{sec:dyn_body}
We provide the dynamics in the body frame with two different set of
coordinates. These dynamics will be helpful in the characterization of the
equilibrium manifold and in the exploration strategy. Indeed, expressing the
dynamics in the body frame, allows us to decouple them from the kinematics and,
thus, write a reduced model which includes only velocities and accelerations.

Since the dynamics do not depend on the positions $x$ and $y$, and the
orientation $\psi$, we can work directly with the longitudinal velocity $v_x$
and the lateral velocity $v_y$. To do this, note that
\begin{equation} \label{eq:ddxddy_dvxdvy}
\left[
  \begin{array}{c}
    \ddot{x} \\
    \ddot{y} \\
  \end{array}
\right] = R_z(\psi) \left[
  \begin{array}{c}
    \dot{v}_x-v_y\dot{\psi} \\
    \dot{v}_y+v_x\dot{\psi} \\
  \end{array}
\right].
\end{equation}
Thus, we get the equations in \eqref{sys:v}.

One more version of the dynamics is obtained by choosing as states the vehicle
speed $v$ and the vehicle sideslip angle $\beta$, where $\tan \beta =
v_y/v_x$. This change of coordinates is helpful to calculate the equilibrium
manifold in the next section.  In this case,
denoting $\chi=\psi+\beta$ the orientation of the velocity with respect to the
spatial frame, we have
\[
\left[
  \begin{array}{c}
    \ddot{x} \\
    \ddot{y} \\
  \end{array}
\right] = R_z(\chi) \left[
  \begin{array}{c}
    \dot{v} \\
    v\dot{\chi} \\
  \end{array}
\right] = R_z(\psi)R_z(\beta) \left[
  \begin{array}{c}
    \dot{v} \\
    v\dot{\chi} \\
  \end{array}
\right],
\]
where $\dot{v}$ and $v\dot{\chi}$ are the longitudinal and lateral
accelerations, respectively. Finally, considering the relation~\eqref{eq:ddxddy_dvxdvy}
we have
\[
\left[
  \begin{array}{c}
    \dot{v} \\
    v\dot{\beta} \\
  \end{array}
\right] = R_z(\beta)^T \left[
  \begin{array}{c}
    \dot{v}_x-v_y\dot{\psi} \\
    \dot{v}_y+v_x\dot{\psi} \\
  \end{array}
\right] - \left[
  \begin{array}{c}
    0 \\
    v\dot{\psi} \\
  \end{array}
\right],
\]
so that the equations of motion are the one given in \eqref{sys:beta}.

We have a family of car models, 
\eqref{sys:v} and \eqref{sys:beta},
that provide different insights depending on the
features to investigate. The model \eqref{sys:v} is used to explore the
dynamics of the car vehicle; the models \eqref{sys:v} and
\eqref{sys:beta} are used to solve the equilibrium manifold (under usual
driving conditions, it is natural to specify $v$ and $\beta$).

\begin{figure*}[!t]
  \normalsize
    \begin{equation} \label{sys:v} 
    \footnotesize{
        \begin{bmatrix}
          m & 0 & 0 & \mu_{fx} & \mu_{rx} \\
          0 & m & mb & \mu_{fy} & \mu_{ry} \\
          0 & mb & (I_{zz}+mb^2) & (a+b)\mu_{fy} & 0 \\
          0 & 0 & 0 & -1 & -1 \\
          -mh & 0 & 0 & a+b & 0 \\
        \end{bmatrix}
        \left[
          \begin{array}{c}
            \dot{v}_x \\
            \dot{v}_y \\
            \ddot{\psi} \\
            f_{fz} \\
            f_{rz} \\
          \end{array}
        \right]
        +
        \begin{bmatrix}
          -mb\dot{\psi}^2-mv_y\dot{\psi} \\
          mv_x\dot{\psi} \\
          mbv_x\dot{\psi} \\
          0 \\
          (I_{xz}+mhb)\dot{\psi}^2+mhv_y\dot{\psi} \\
        \end{bmatrix}
        +
        \begin{bmatrix}
          0 \\
          0 \\
          0 \\
          -mg \\
          mgb \\
        \end{bmatrix}
        =
        \begin{bmatrix}
          0 \\
          0 \\
          0 \\
          0 \\
          0 \\
        \end{bmatrix}
      }
    \end{equation}
    \begin{equation} \label{sys:beta} 
    \footnotesize{
        \begin{bmatrix}
          mc_{\beta} & -mvs_{\beta} & 0 & \mu_{fx} & \mu_{rx} \\
          ms_{\beta} & mvc_{\beta} & mb & \mu_{fy} & \mu_{ry} \\
          mbs_{\beta} & mbvc_{\beta} & (I_{zz}+mb^2) & (a+b)\mu_{fy} & 0 \\
          0 & 0 & 0 & -1 & -1 \\
          -mhc_{\beta} & mhvs_{\beta} & 0 & a+b & 0 \\
        \end{bmatrix}
        \!
        \left[
          \begin{array}{c}
            \dot{v} \\
            \dot{\beta} \\
            \ddot{\psi} \\
            f_{fz} \\
            f_{rz} \\
          \end{array}
        \right]
        \! + \!
        \begin{bmatrix}
          -mv\dot{\psi}s_{\beta}-mb\dot{\psi}^2\\
          -mv\dot{\psi}c_{\beta} \\
          mbv\dot{\psi}c_{\beta} \\
          0 \\
          (I_{xz}+mhb)\dot{\psi}^2+mhv\dot{\psi}s_{\beta} \\
        \end{bmatrix}
        \! + \!
        \begin{bmatrix}
          0\\
          0 \\
          0 \\
          -mg \\
          mgb \\
        \end{bmatrix}
        \! = \!
        \begin{bmatrix}
          0\\
          0 \\
          0 \\
          0 \\
          0 \\
        \end{bmatrix}
      }
    \end{equation}
\end{figure*}

\begin{remark}[Model development and existing literature]
  The proposed LT-CAR model development differs form the one proposed in
  \cite{EV-PT-JL:08} as follows. First, we provide a detailed derivation of the
  model based on the Lagrangian approach. This derivation allows us to exploit
  an interesting structure of the proposed car vehicle.
  Second, we consider the off-diagonal inertia term $I_{xz}$, see the Coriolis
  term in \eqref{sys:v} and \eqref{sys:beta}, which becomes significant in
  studying both cornering equilibria and aggressive maneuvers.
  Third and final, we analyze the region of model validity in terms of the
  vehicle (geometric and tire) parameters. \oprocend
\end{remark}

\section{Equilibrium Manifold}
\label{sec:eq_manif}
In this section we analyze the equilibrium manifold of the car model, i.e. the
set of trajectories that can be performed by use of constant inputs.
Searching for ``constant'' trajectories requires the solution of a set of
nonlinear equations expressing the fact that all accelerations must be set to
zero.  To define an equilibrium trajectory, we refer to the car model in the
form \eqref{sys:beta}. The equilibria are obtained by enforcing
\begin{equation}
  (\dot{v},\dot{\beta},\ddot{\psi})=(0,0,0).
  \label{eq:constr}
\end{equation}
The corresponding trajectory of the full car model (including position and
orientation) is a circular path at constant speed $v$, yaw rate $\dot{\psi}$ and
vehicle sideslip angle $\beta$. Since $\dot{\beta}=0$, the lateral acceleration
is given by $a_{lat}=v\dot{\psi}$, and expressing the accelerations in the body
frame as follows,
\[
\left[
  \begin{array}{c}
    a_x \\
    a_y \\
  \end{array}
\right] = \left[
  \begin{array}{c}
    \dot{v}_x - v_y\dot{\psi} \\
    \dot{v}_y + v_x\dot{\psi} \\
  \end{array}
\right] = R_z(\beta) \left[
  \begin{array}{c}
    \dot{v} \\
    v\dot{\chi} \\
  \end{array}
\right]
\]
we have
\[
  a_x = -a_{lat}\sin{\beta}\,, \quad
  a_y = a_{lat}\cos{\beta}\,, \quad
  \dot{\psi} = a_{lat}/v.
\]
Now, referring to the dynamic model (\ref{sys:beta}), we set the
constraints (\ref{eq:constr}) and we get two equations from the load transfer in
equilibrium condition
\begin{equation}
\begin{split}
  -f_{fz} &= mg\frac{b}{a+b}+\frac{(I_{xz}+mhb)(\frac{a_{lat}}{v})^2+a_{lat}mh\sin\beta}{a+b} \\
  -f_{rz} &= mg\frac{a}{a+b}-\frac{(I_{xz}+mhb)(\frac{a_{lat}}{v})^2+a_{lat}mh\sin\beta}{a+b}
\end{split}
  \label{eq:eqmfd_fz}
\end{equation}
and the following three equations from the system dynamics:
\begin{equation}
  \begin{split}
    ma_x-mb\dot{\psi}^2+\mu_{fx}f_{fz}+\mu_{rx}f_{rz}&=0\\
    ma_y+\mu_{fy}f_{fz}+\mu_{ry}f_{rz}&=0\\
    mba_y+(a+b)\mu_{fy}f_{fz}&=0.\\
  \end{split}
  \label{eq:eq_manifold}
\end{equation}
For the clarity of presentation, we perform the equilibrium manifold computation
by using only the rear slip $\kappa_r$ as control input (and setting the
longitudinal one, $\kappa_f$, to zero). Substituting the expression of the
normal forces \eqref{eq:eqmfd_fz} into equations \eqref{eq:eq_manifold}, we
obtain a nonlinear system of three equations in five unknowns ($v$, $a_{lat}$,
$\beta$, $\delta$ and $\kappa_r$), so that the equilibrium manifold is a
two-dimensional surface.
%
%
We parameterize the equilibrium manifold in terms of the car speed and lateral
acceleration ($v$ and $a_{lat}$), so that the slip angle, steer angle and
longitudinal slip ($\beta$, $\delta$ and $\kappa_r$) are obtained by solving the
nonlinear equations in \eqref{eq:eq_manifold}.

We solve the nonlinear system by using a predictor corrector continuation
method, as described in \cite{ELA-KG:90}, relying on the continuity of the equilibria with
respect to the equilibrium manifold parameters $v$ and $a_{lat}$.
Next, we describe the predictor corrector continuation method applied to the
equilibrium manifold of our car model. We fix the
velocity $v$ and explore a one-dimensional slice of the manifold.
First, we provide a useful lemma from \cite{ELA-KG:90}.
\begin{lemma}[Lemma 2.1.3, \cite{ELA-KG:90}]
  \label{lem:num_cont_meth}
  Let $\feqmfd:\real^{n+1} \rightarrow \real^n$ be a smooth nonlinear function
  such that $\feqmfd(\xeqmfd_0)=0$ for some $\xeqmfd_0\in \real^{n+1}$ and let
  the \emph{Jacobian matrix} $D \feqmfd(\xeqmfd_0) \in \real^{n \times (n+1)}$
  have maximum rank. Then, there exists a smooth curve $s \in [0,s_1) \mapsto
  c(s) \in \real^{n+1}$, parametrized with respect to arclength $s$, for some
  open interval $[0,s_1)$ such that for all $s \in [0, s_1)$: i) $c(0) =
  \xeqmfd_0$, ii) $\feqmfd(c(s))=0$, iii) $rank(D \feqmfd(c(s)))=n$, and iv) $\dot{c}(s) \neq 0$. \oprocend
\end{lemma}

Let $\xeqmfd=[a_{lat}, \beta, \delta, \kappa_r]^T$ and let $\feqmfd(\xeqmfd)=0$ be
the nonlinear system in \eqref{eq:eq_manifold}, with $\feqmfd: \real^4
\rightarrow \real^3$. The following proposition shows that there exists a one
dimensional manifold of solution points.

\begin{proposition}[Equilibrium manifold well posedness]
  \label{prop:eq_mnfd}
  Given the nonlinear system in \eqref{eq:eq_manifold}, the following holds
  true:
  \begin{enumerate}
  \item there exists a smooth curve $s \in [0,s_1) \mapsto c(s) \in \real^{4}$,
    for some $s_1>0$, 
    such that $\feqmfd(c(s)) = 0$ for all $s\in [0,s_1)$;
  \item $c(s)$ is the local solution of
    \begin{equation} \label{prop:init_val_pr}
      \begin{split}
        \dot{\xeqmfd} = v^\top(\xeqmfd) \qquad \xeqmfd(0) = \xeqmfd_0,
      \end{split}
    \end{equation}
    where $v^\top(\xeqmfd)$ is the \emph{tangent vector induced by}
    $D \feqmfd(\xeqmfd)$.
  \end{enumerate} 
\end{proposition}

\begin{proof}
  To prove statement (i), we use Lemma~\ref{lem:num_cont_meth}.
  The nonlinear function $\feqmfd$ contains sums and products of
  trigonometric and power functions, thus it is smooth.  Using the expression of
  the combined slip forces introduced in Section~\ref{sec:tires_gen_forces}, for
  $\xeqmfd_0=[0,0,0,0]^T$ we have $\mu=0$, so that 
  $\feqmfd(\xeqmfd_0)=0$. Moreover, by explicit calculation, the Jacobian matrix
  at $\xeqmfd_0$ has rank three.

  To prove statement (ii), we differentiate $\feqmfd(c(s))=0$ with respect to
  the arc-length $s$. The tangent $\dot{c}(s)$ satisfies 
  $D \feqmfd(c(s))\dot{c}(s)=0, \,\,\, \|\dot{c}(s)\|=1 \;\; \forall s\in
  [0,s_1)$.
  Hence $\dot{c}(s)$ spans the one-dimensional kernel $ker(D \feqmfd(c(s)))$, or
  equivalently, $\dot{c}(s)$ is orthogonal to all rows of $D \feqmfd(c(s))$.
  In other words, the unique vector $\dot{c}(s)$ is the tangent vector induced
  by $D \feqmfd(c(s))$, $v^\top(\xeqmfd)$. Using the Implicit Function
  Theorem, e.g., \cite{MWH:97}, 
  the tangent vector $v^\top(\xeqmfd)$ depends smoothly on
  $D \feqmfd(c(s))$. Thus, $c$ is the solution curve of the initial value
  problem in \eqref{prop:init_val_pr}, which concludes the proof.
\end{proof}


%
In order to numerically trace the curve $c$ efficiently, we use a
predictor-corrector method.  The main idea is to generate a sequence of
points along the curve $\xeqmfd_i$, $i=1,2,\ldots$, that satisfy a given
tolerance, say $\|\feqmfd(\xeqmfd_i)\| \leq \nu$ for some $\nu > 0$.
So, for $\nu>0$ sufficiently small, there is a unique parameter value $s_i$ such
that the point $c(s_i)$ on the curve is nearest to $\xeqmfd_i$ in Euclidean
norm.  To describe how points $\xeqmfd_i$ along the curve $c$ are generated,
suppose that a point $\xeqmfd_i \in \real^4$ satisfies the chosen tolerance
(i.e. $\|\feqmfd(\xeqmfd_i)\| \leq \nu$). If $\xeqmfd_i$ is a regular point of
$\feqmfd$, then there exists a unique solution curve $c_i:[0,s_1) \rightarrow
\real^4$ which satisfies the initial value problem \eqref{prop:init_val_pr} with
initial condition $\xeqmfd(0) = \xeqmfd_i$.

To obtain a new point $\xeqmfd_{i+1}$ along $c$, we make a \emph{predictor
  step} as a simple numerical integration step for the initial value problem.
We use an \emph{Euler  predictor}:
  $\alpha_{i+1} = \xeqmfd_i + \epsilon \,v^\top(\xeqmfd)$, 
where $\epsilon>0$ represents a suitable stepsize.
%
The corrector step computes the point $\omega_{i+1}$ on $c$ which is nearest
to $\alpha_{i+1}$. The point $\omega_{i+1}$ is found by solving the optimization
problem
\begin{equation} \label{eq:correction}
  \|\omega_{i+1}-\alpha_{i+1}\|=\min_{\feqmfd(\omega)=0}\|\omega-\alpha_{i+1}\|
  \, .
\end{equation}
If the stepsize $\epsilon$ is sufficiently small (so that the predictor point
$\alpha_{i+1}$ is sufficiently close to the curve $c$) the minimization problem
has a unique solution $\omega_{i+1}$. We compute $\omega_{i+1}$ by using a
Newton-like method. The \emph{Newton point} $\NN(\alpha)$ for approximating the
solution of \eqref{eq:correction} is given by
$\NN(\alpha) = \alpha - D \feqmfd(\alpha)^\dag \feqmfd(\alpha)$.
%


The predictor-corrector continuation method used in the paper thus consists of repeatedly
performing these predictor and corrector steps as shown in the pseudo-code below.
\vspace{-2ex}
\begin{center}
  \begin{minipage}{0.7\linewidth}
\begin{algorithm}[H]
  \caption{Predictor-corrector continuation method}
  \label{alg:continuation_method}
  \begin{algorithmic}
    \REQUIRE initial equilibrium condition $\xeqmfd_0$ such that
    $\feqmfd(\xeqmfd_0)=0$ \FOR{$i = 0, 1, 2 \ldots$} \STATE set the initial
    steplength $\epsilon_i=\overline{\epsilon}$;
    \LOOP \STATE get predictor step: $\alpha_{i+1} = \xeqmfd_i + \epsilon_i \,
    v^\top(\xeqmfd)$;
    \STATE search corrector term:\\
    $~~~~~~~~~~$    $\omega_{i+1} = \alpha_{i+1} - D \feqmfd(\alpha_{i+1})^\dag \feqmfd(\alpha_{i+1})$;\\
    $~~~~~~~~~~$ $\alpha_{i+1}=\omega_{i+1}$;
    \STATE \textbf{if} convergence \textbf{then} break;
    \STATE \textbf{else} update step-length $\epsilon_{i+1} = \frac{\epsilon_{i}}{2}$;\\
    \textbf{end if}
    \ENDLOOP
    \STATE $\xeqmfd_{i+1} = \omega_{i+1}$;
    \ENDFOR
  \end{algorithmic}
\end{algorithm}
\end{minipage}
\end{center}

We compute and compare the equilibrium manifold for the car model with and
without load transfer (i.e. LT-CAR and bicycle model). The model parameters are
the one of a sports car with rear-wheel drive transmissions given in
Appendix~\ref{APP:params}.

\begin{figure*}[ht!]
  \begin{center}
  \!\subfloat[]{\includegraphics[scale=.28]{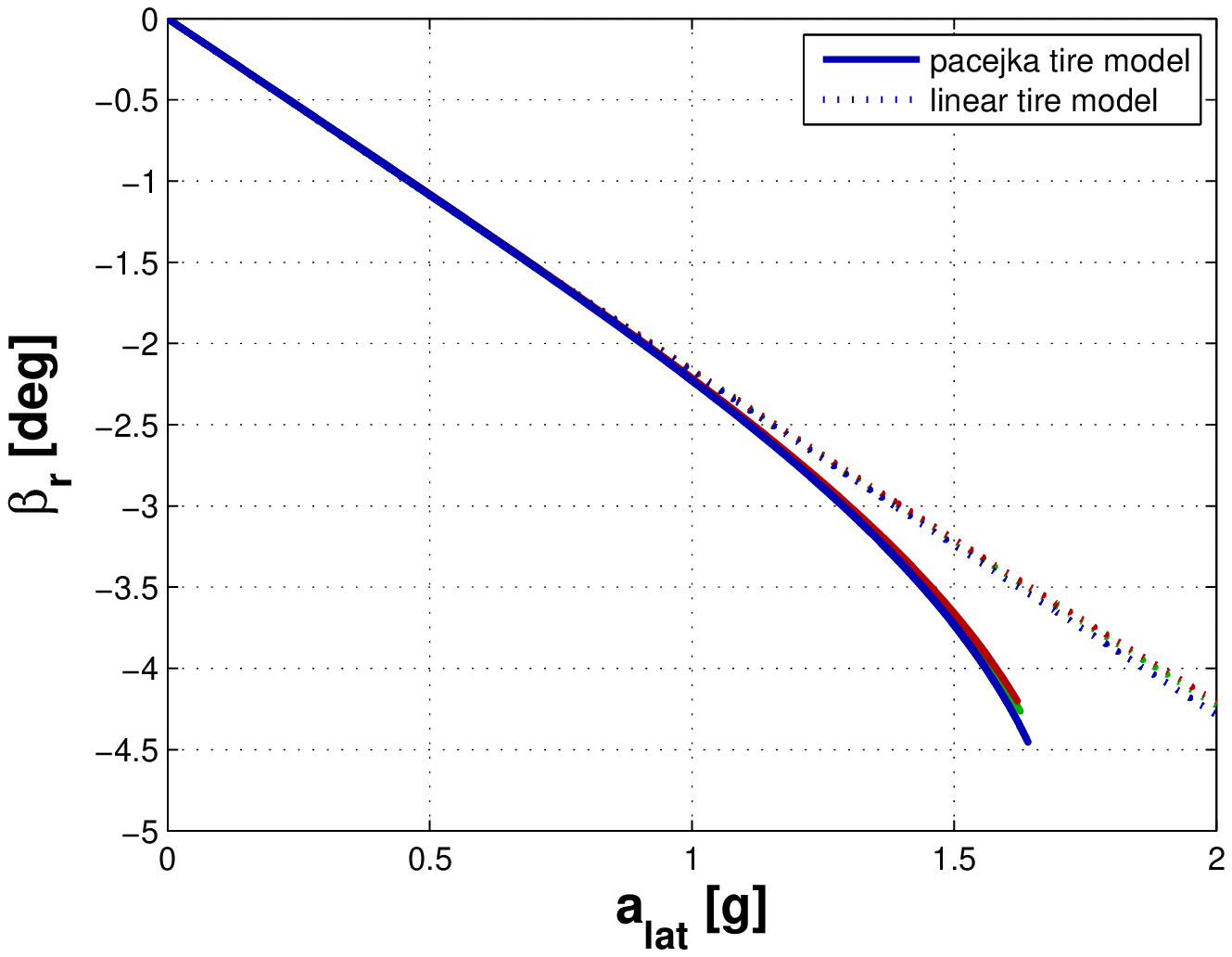} \label{fig:eqmf_fz_betar}}
  \!\!\!\!\!\subfloat[]{\includegraphics[scale=.28]{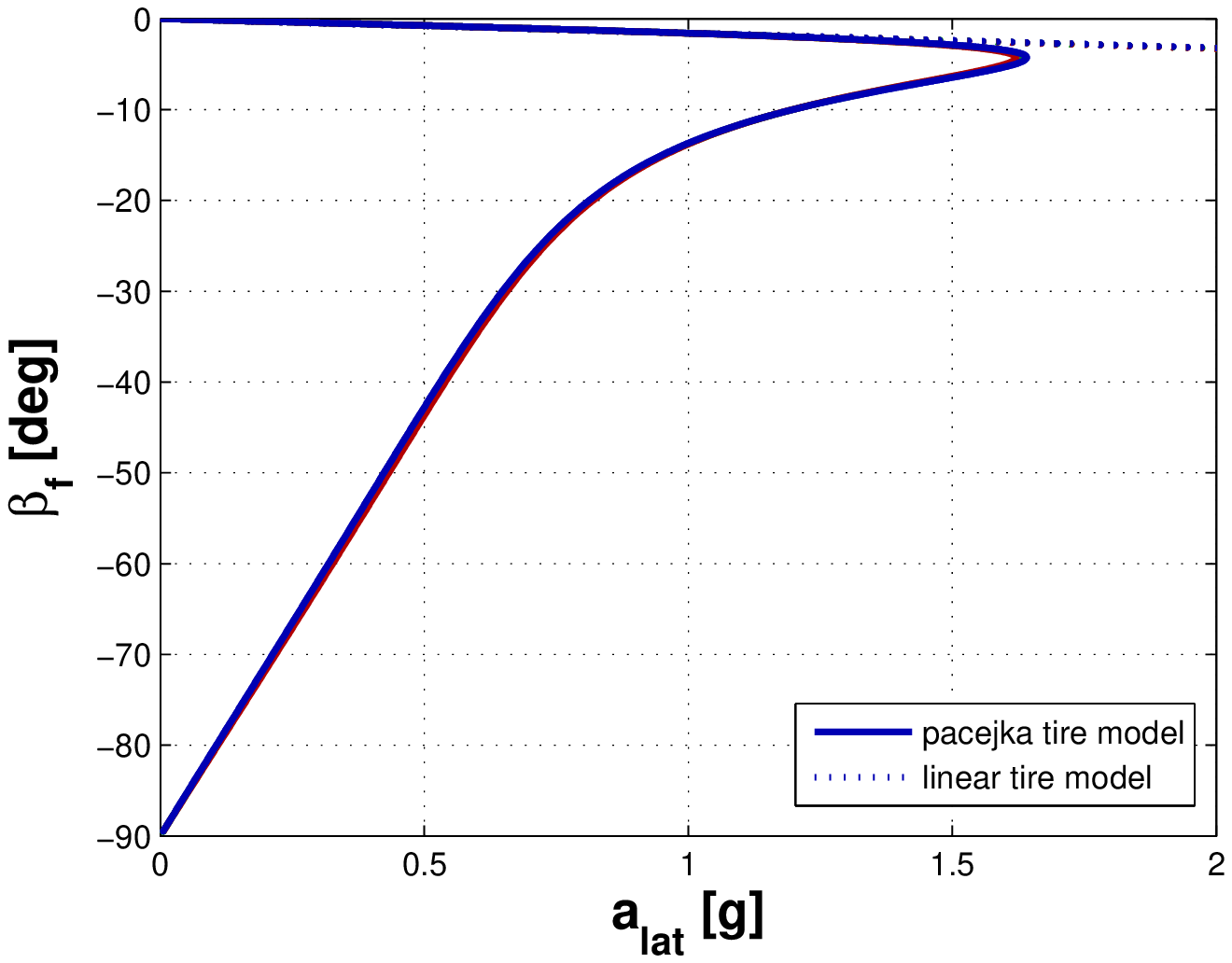} \label{fig:eqmf_fz_betaf}}
  \!\!\!\!\!\subfloat[]{\includegraphics[scale=.28]{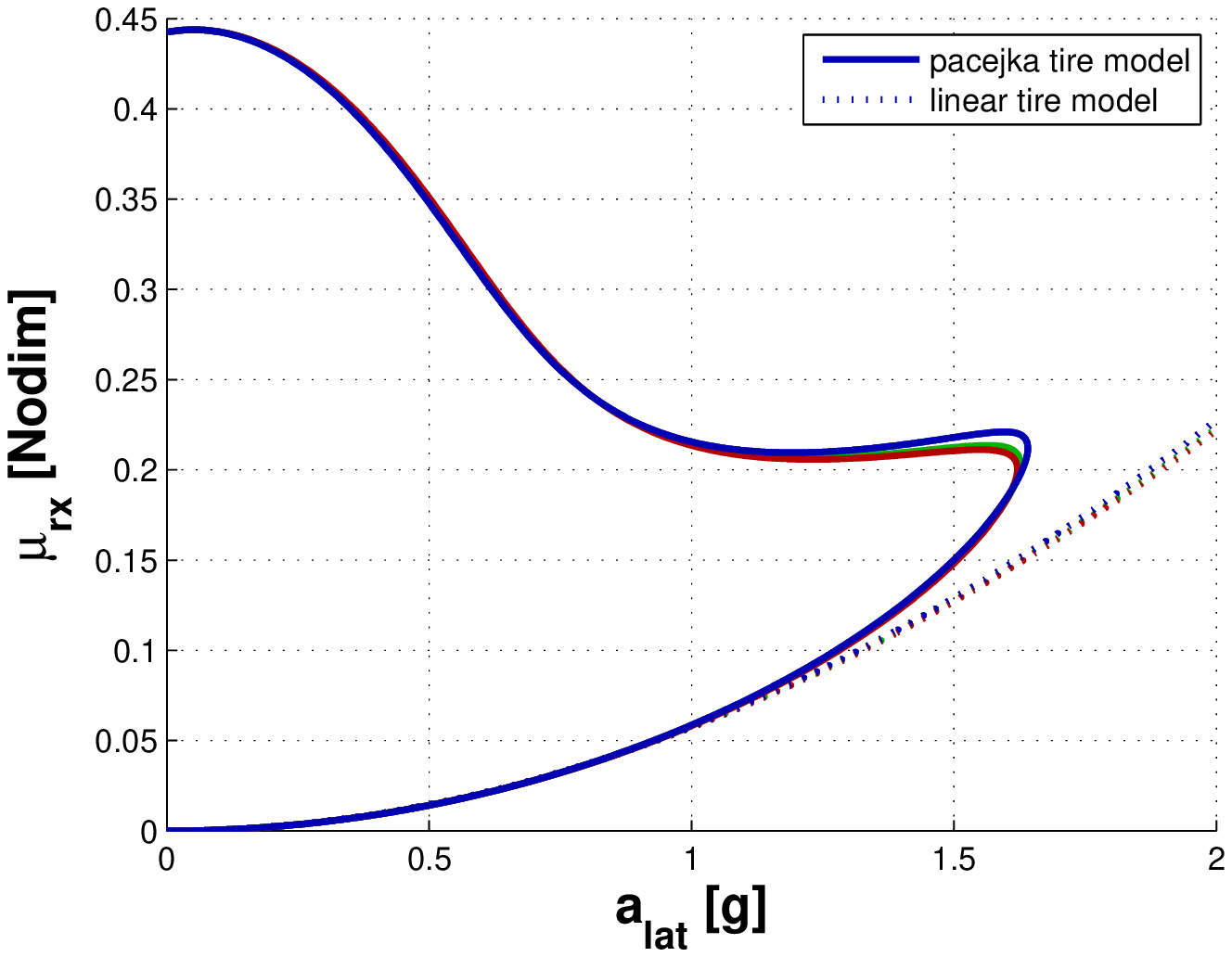} \label{fig:eqmf_fz_murx}}
  \!\subfloat[]{\includegraphics[scale=.28]{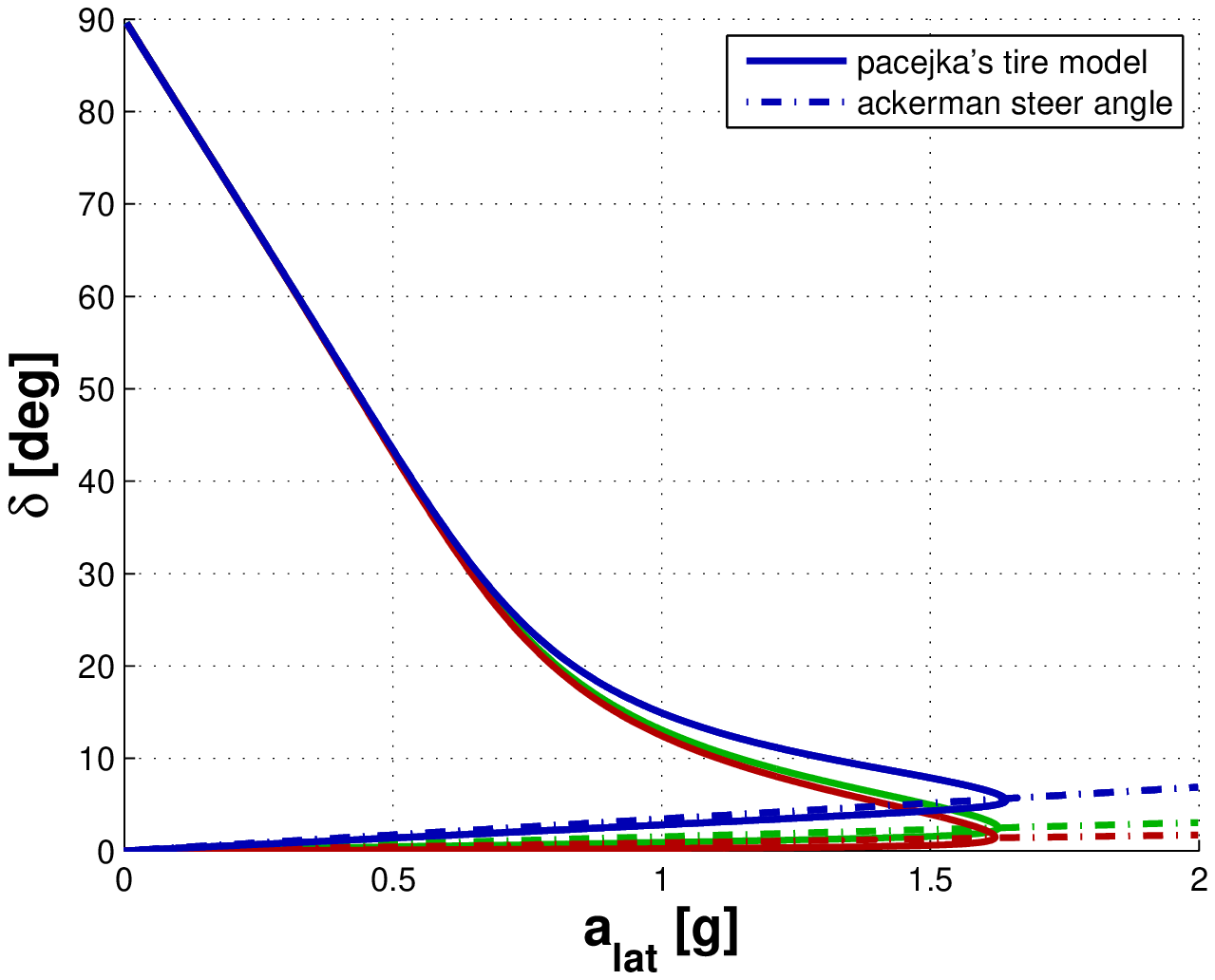} \label{fig:eqmf_fz_delta}}
%

  \!\subfloat[]{\includegraphics[scale=.28]{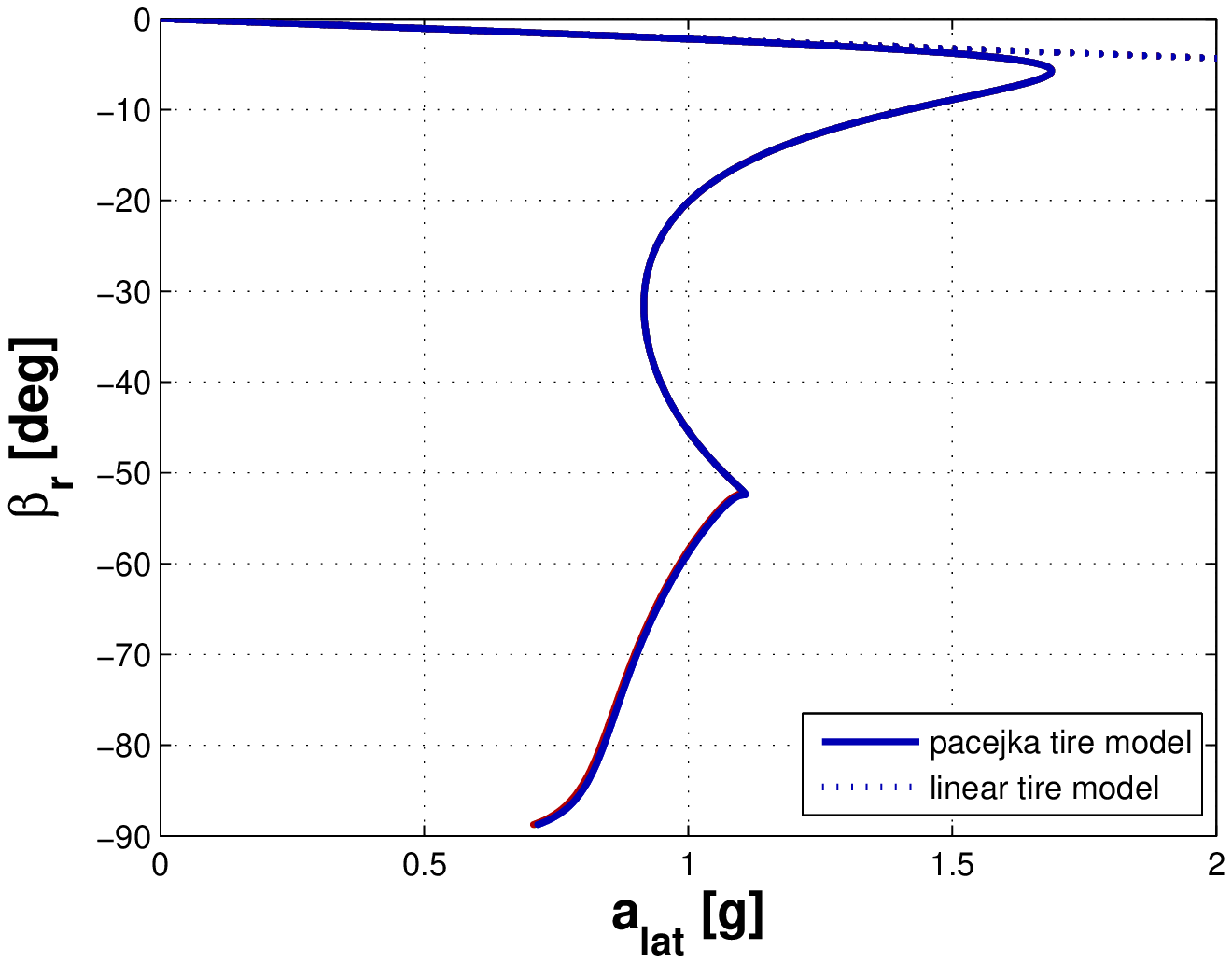} \label{fig:eqmf_nofz_betar}}
  \!\!\!\!\!\subfloat[]{\includegraphics[scale=.28]{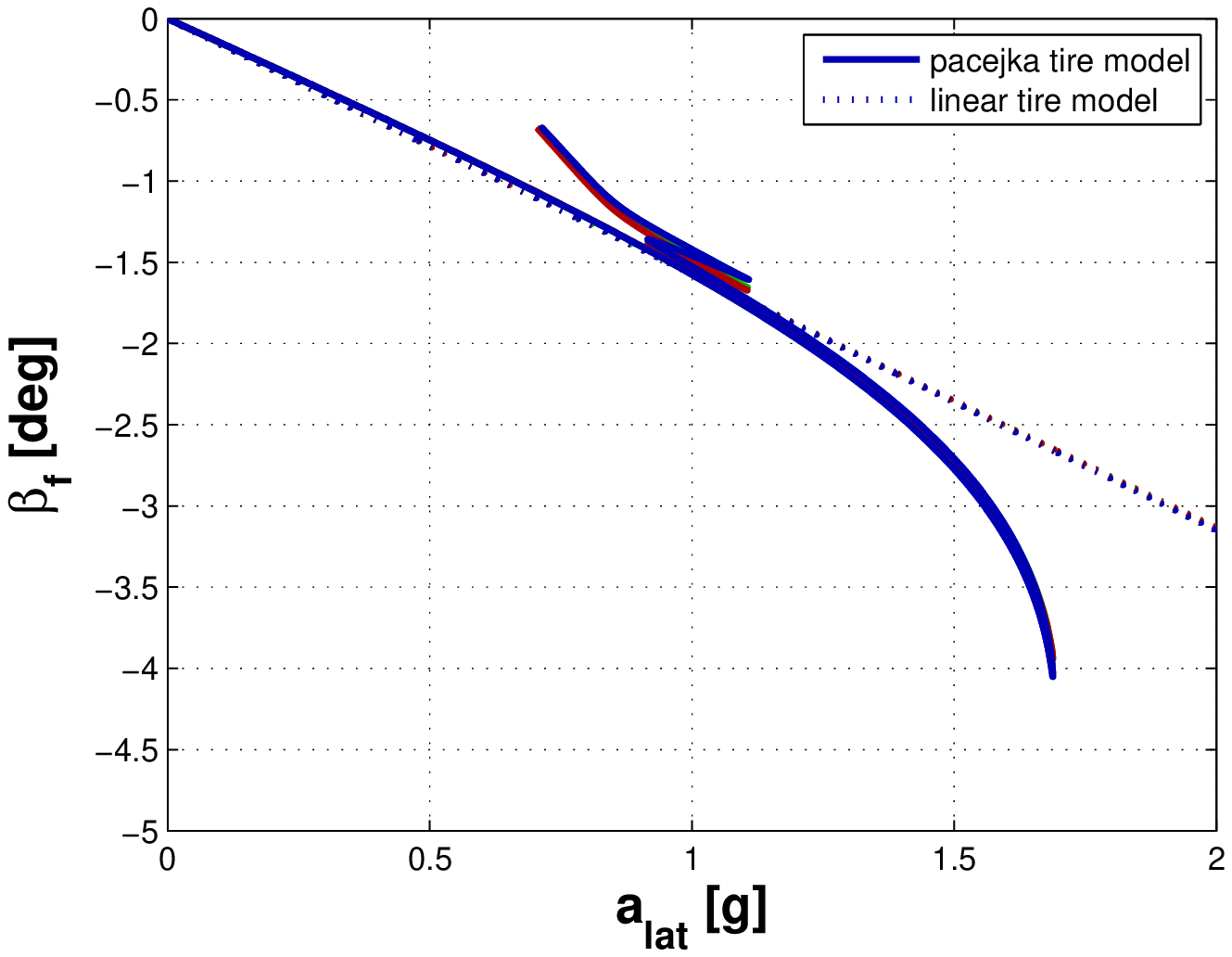} \label{fig:eqmf_nofz_betaf}}
  \!\!\!\!\!\subfloat[]{\includegraphics[scale=.28]{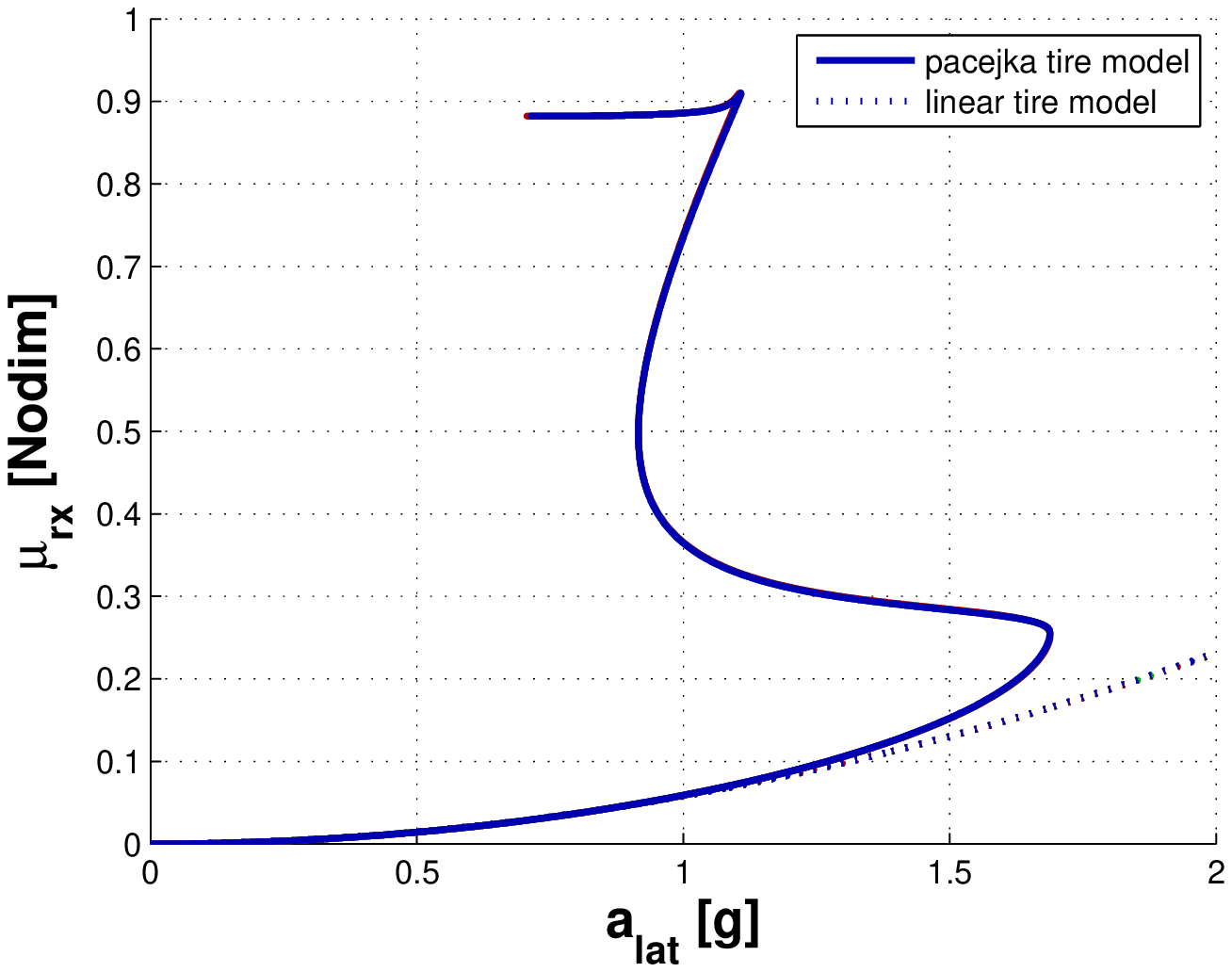} \label{fig:eqmf_nofz_murx}}
  \!\subfloat[]{\includegraphics[scale=.28]{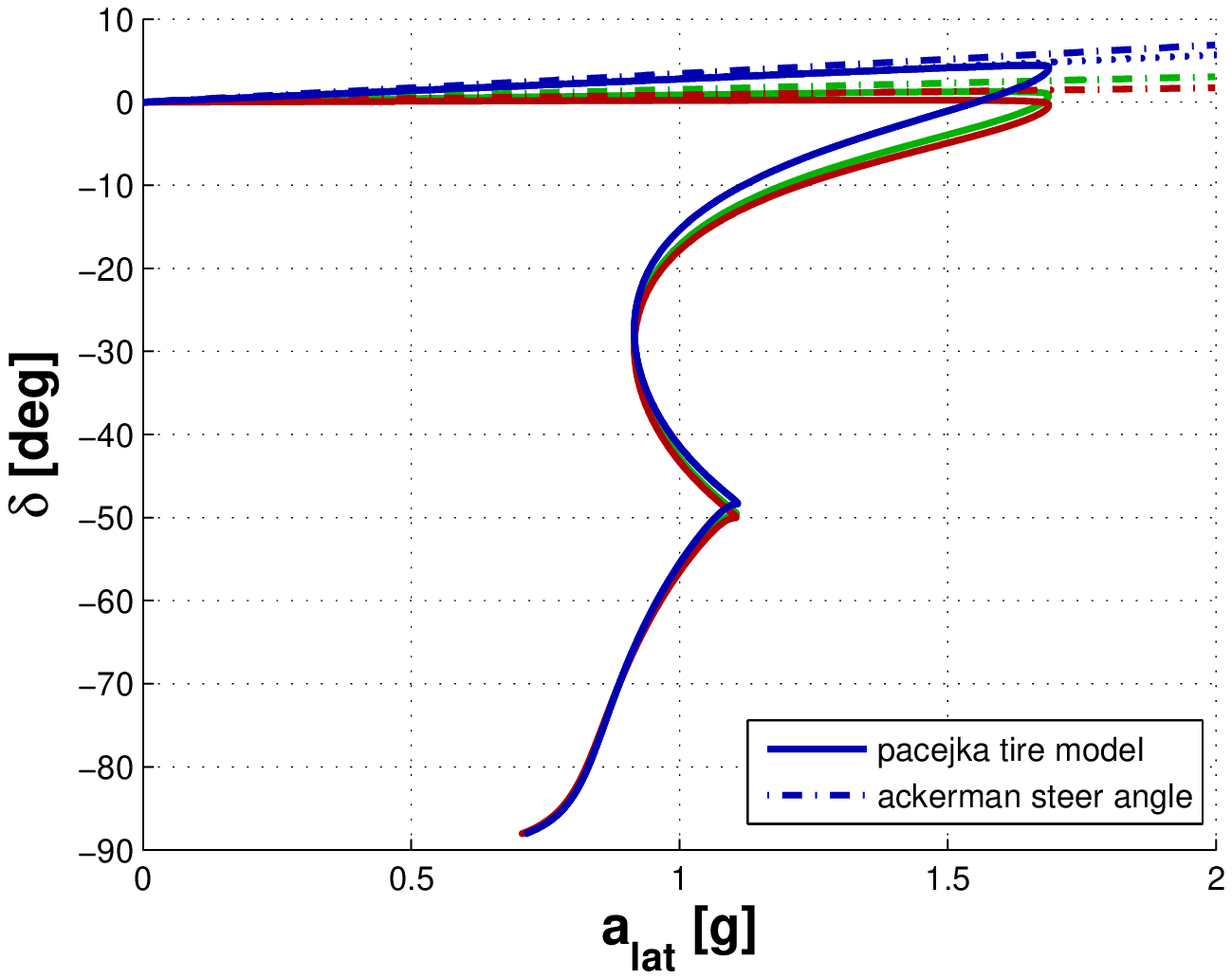} \label{fig:eqmf_nofz_delta}}
  \caption{\small Equilibrium manifold for a rear-wheel drive sports car
      with longitudinal load transfer in (a)-(d), and without load transfer in
      (e)-(h). Specifically: rear and front sideslip, longitudinal force coefficient, and steer angle for $v = (20, 30,
      40)$ m/s.  Dot lines in (a)-(b)-(c) and (e)-(f)-(g) are the equilibria
      with linear tire model. The dash-dot line in (d) and (h) is the Ackerman
      steer angle.}
    \label{fig:eqmf_berspo}
%

    \!\subfloat[]{\includegraphics[scale=.28]{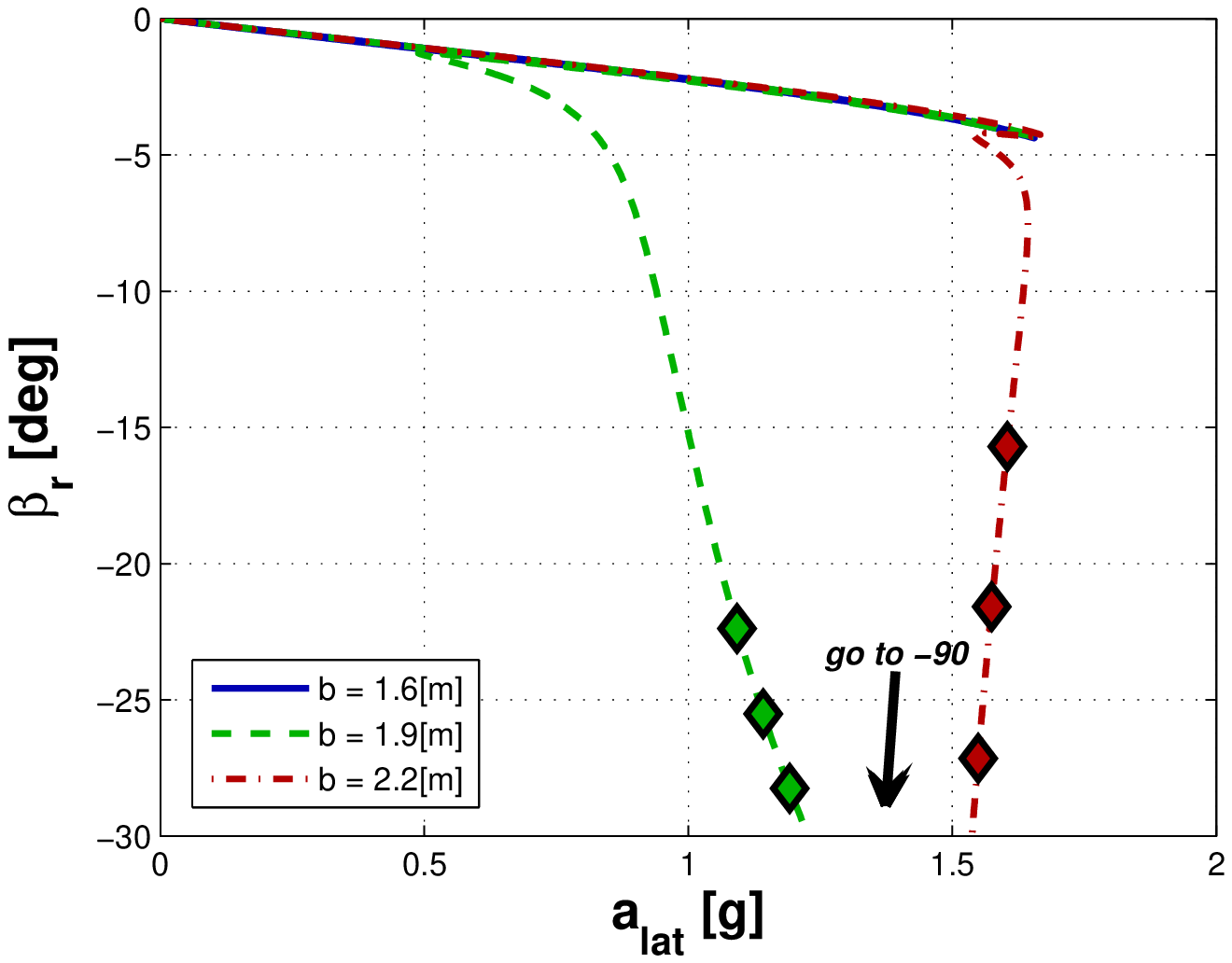} \label{fig:eqmf_varb_betar}}
    \!\!\!\!\!\subfloat[]{\includegraphics[scale=.28]{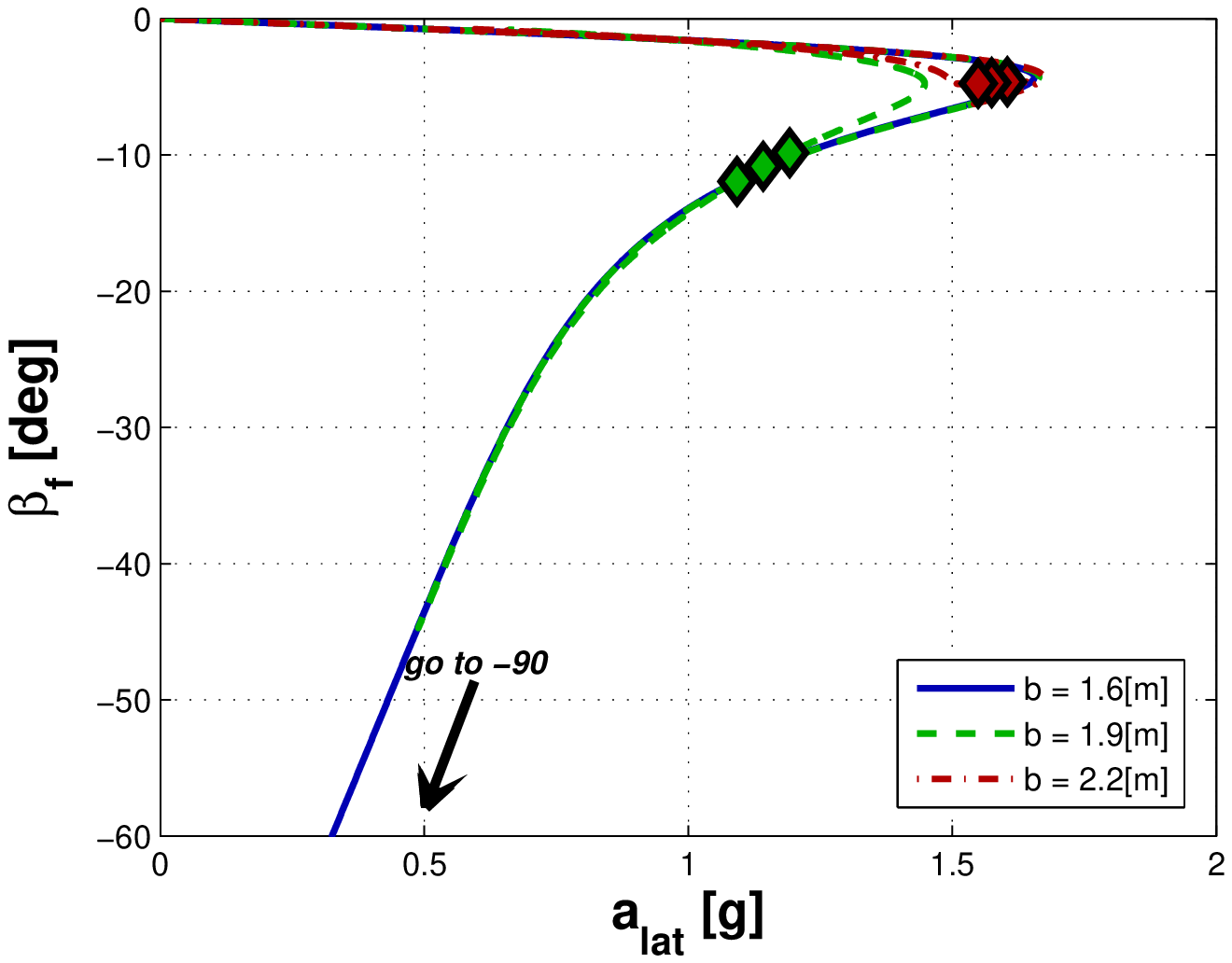} \label{fig:eqmf_varb_betaf}}
    \!\subfloat[]{\includegraphics[scale=.28]{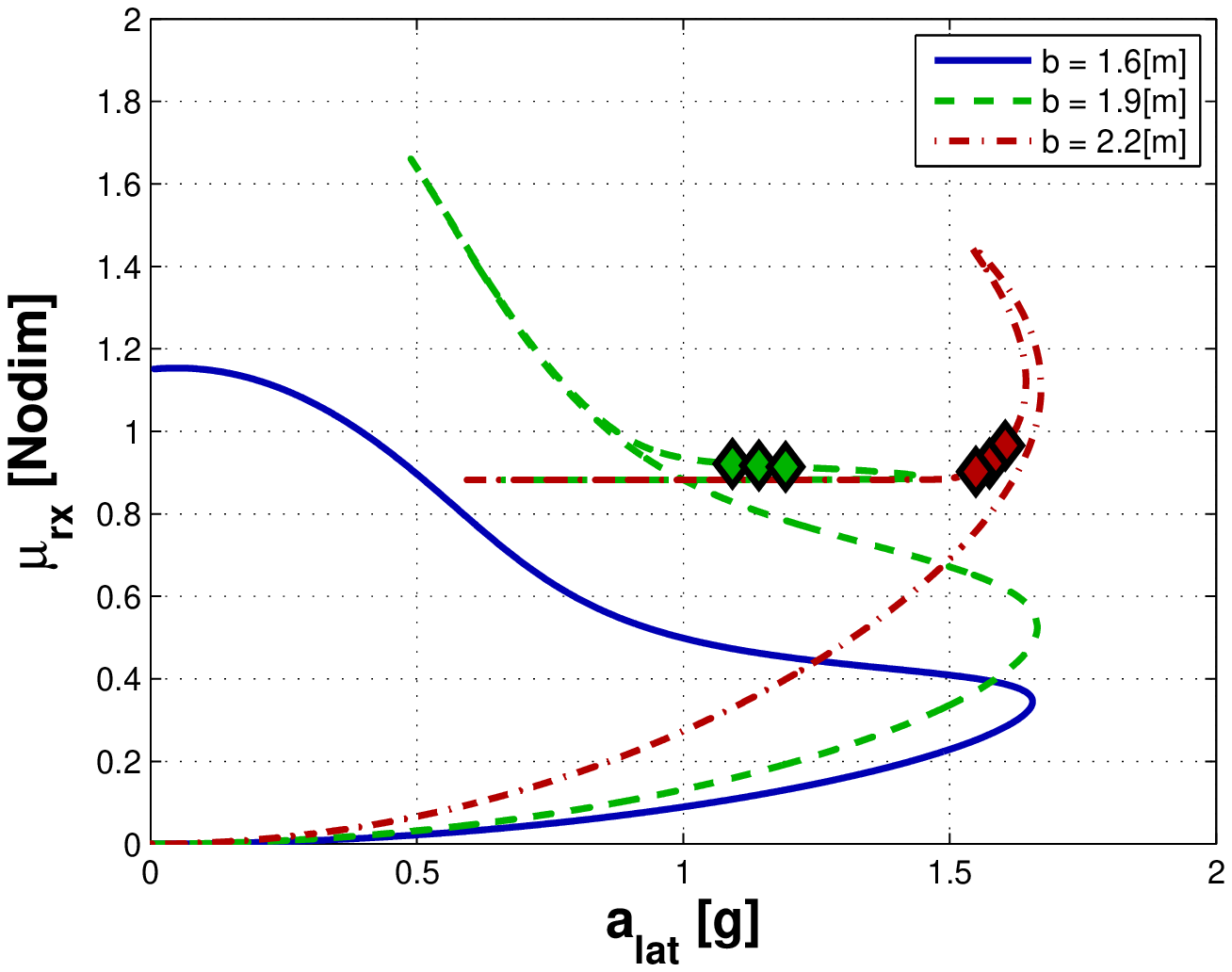} \label{fig:eqmf_varb_murx}}
    \!\!\!\!\!\subfloat[]{\includegraphics[scale=.28]{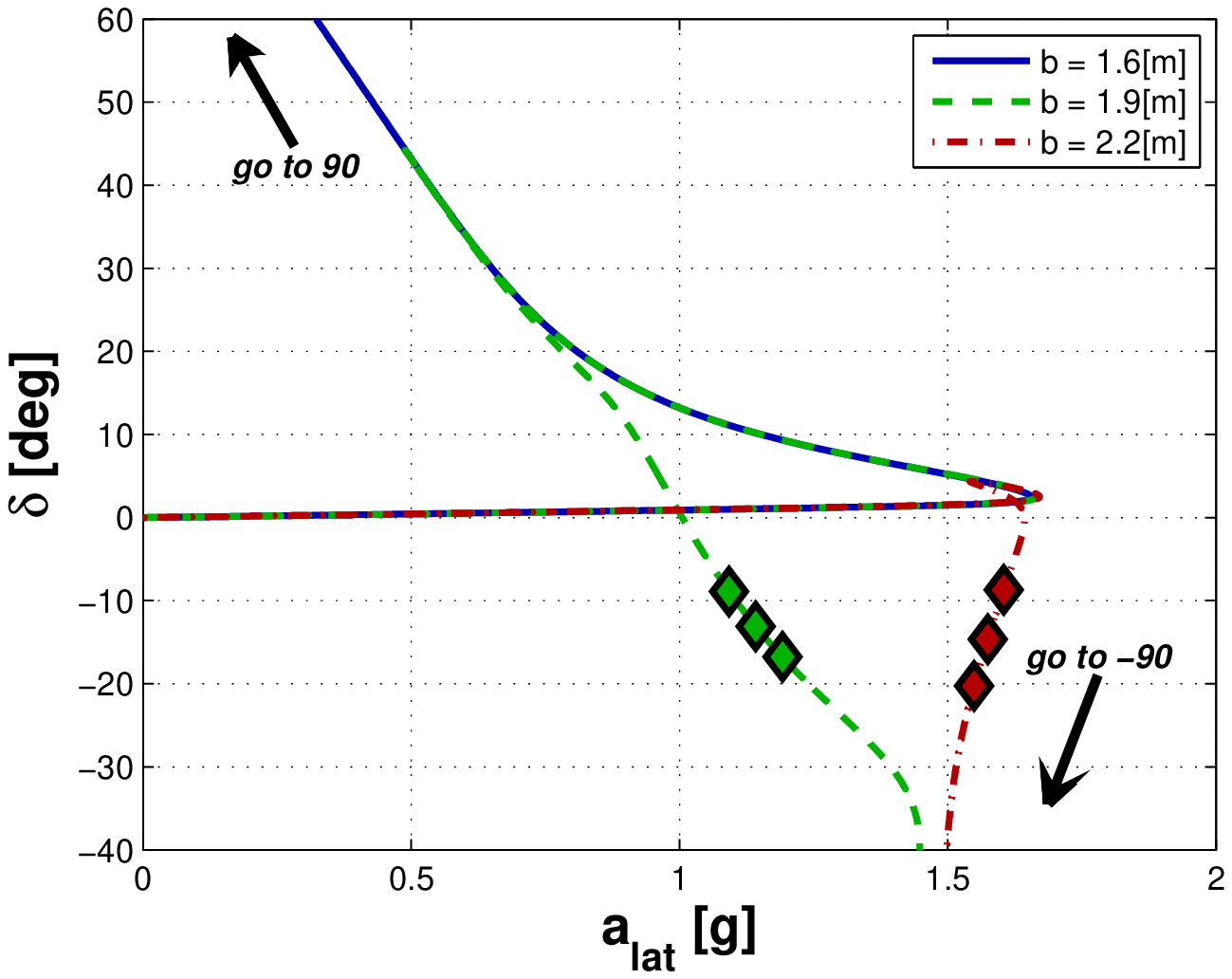} \label{fig:eqmf_varb_delta}}
    \caption{\small Equilibrium manifold for a rear-wheel drive
    sports car for different positions of the center of mass. Specifically: rear
      sideslip, front sideslip, longitudinal force coefficient and steer
      angle for $v = 30$ m/s, $b=(1.6,1.9,2.1)$ m and $a+b=2.45$ m. In (d)
      the red and green diamond markers show three equilibrium points with steer
      angle opposite to the direction of the turn (counter-steering).
    }
    \label{fig:eqmfd_varying_b}
%
%
%
  \end{center}
\vspace{-2ex}
\end{figure*}

Some slices of the equilibrium manifold are shown in
Figure~\ref{fig:eqmf_berspo}.
%
The plots are given only for positive values of the lateral acceleration due to
the symmetry of the equilibrium manifold. Indeed, the rear and front sideslip,
and the steer angle are symmetric functions of $a_{lat}$, while the
longitudinal force coefficient is antisymmetric.

For low longitudinal and lateral slips a first
class of equilibria appears. These equilibria are close to the ones with the
linear tire approximation (the solid lines in Figures~\ref{fig:eqmf_berspo} are close to the dot lines).
Indeed, for low slips ($\beta_r,\beta_f<5$ deg and $\kappa_r < 0.05$) the
tires work within their linear region as appears in Figures~\ref{fig:pacejka_long}~and~\ref{fig:pacejka_lat}.
To characterize the vehicle behavior in this region, we can use,
\cite{TG:92}, the \emph{understeer gradient}
\[
  K_{us}(a_{lat};v) = \frac{\partial\delta(a_{lat};v)}{\partial a_{lat}} - K_a,
\]
where $K_a = \frac{a+b}{v^2}$ is called Ackerman steer angle gradient. The
vehicle is said to be understeering if $K_{us}>0$, neutral if $K_{us}=0$ and
oversteering if $K_{us}<0$.
From a graphical point of view, the understeering behavior can be measured by
looking at how much the curve $a_{lat}\mapsto\delta(a_{lat};v)$, for fixed
$v$, departs from the line $a_{lat}\mapsto K_a a_{lat}$.
%
As shown in Figure~\ref{fig:eqmf_fz_delta} and \ref{fig:eqmf_nofz_delta}, the steer angle
gradient is slightly negative, which suggests an oversteering behavior $K_{us} < 0$.
It is worth noting that if the relationship between $a_{lat}$ and $\beta$ ($=\beta_r$) is approximately linear, then the load transfer is approximately quadratic as a function of $a_{lat}$ by \eqref{eq:eqmfd_fz}. This is in fact what occurs in this case (see Figure~\ref{fig:eqmf_fz_betar}), so that $-f_{rz}$ increases approximately quadratically with $a_{lat}$ (with $-f_{fz}$ decreasing by the same amount). 
That is, with increased lateral acceleration, we see that the rear tire(s) becomes more effective (due to increased loading) and the front becomes less effective. 

For high values of the longitudinal and lateral slips the equilibrium manifolds
depart from their linear-tire approximation. 
Indeed, the linear tires, without
force saturation, can generate a lateral force for a wider range of lateral
acceleration. 
We observe a significantly different structure of the equilibrium manifold for
the two models, which gives a first evidence of the importance of taking into
account the load transfer phenomenon.
Specifically, since the available tire force is limited for the LT-CAR model, the achievable lateral acceleration is also limited. This limit occurs as a smooth turning of the equilibrium manifold as seen in, e.g., Figures~\ref{fig:eqmf_fz_betaf} and ~\ref{fig:eqmf_fz_murx}, and the manifold can be continued with steering going all the way to $90$[deg] at which point the front tire(s) is just pushing with no lateral force. 
The steady-state behavior of the LT-CAR is clear: for a given $v$ and $a_{lat}$, the required $\beta$ is close to that predicted by the linear tire model while there are two solutions (except at the maximum value of $a_{lat}$) for the steering angle providing rather different values of $\beta_f$ and $\kappa_r$. 
This straightforward behavior of $\beta$ seems to be related to the fact that the rear tire(s) becomes more effective as $a_{lat}$ is increased. 
In contrast, for the bicycle model (i.e., without load transfer) the front tire(s) retains a higher loading and its equilibrium manifold shows that it is the front sideslip $\beta_f$ that keeps a small angle (so that the front wheel is nearly aligned with the direction of travel) while the rear sideslip is allowed to be quite large, see Figures~\ref{fig:eqmf_nofz_betar} and~\ref{fig:eqmf_nofz_betaf}. 
%
%
%
Topologically, the equilibrium manifolds of the system with and without load transfer are strikingly different. 
%
Clearly, the load
transfer is the responsible of this difference due to its modulating action on
the ground forces.

The predictor-corrector continuation method can be also used to perform a
sensitivity analysis of the equilibrium manifold with respect to the car
parameters (as, e.g., mass, moment of inertia, center of mass position).
In Figure~\ref{fig:eqmfd_varying_b} we highlight the results obtained
when varying the center of mass position along the body longitudinal axis. By
setting the sports car inertial parameters, we compute the manifold varying
the value of $a$ and $b$ with constant wheelbase $a+b=2.45$ m.
%
 %
%
When the center of mass is moved over the half wheelbase toward the front
axle, the manifold has a significantly different structure (green and red
lines). In particular, the equilibria at highest rear lateral and
longitudinal slips, highlighted with the red diamond markers, are achieved
with steer angle opposite to the direction of the turn (counter-steering).
This car set up resembles the one of rally cars which, indeed, take advantage
of the counter-steering behavior in performing aggressive turns.
The significant change of the equilibrium manifold with respect to the position
of the center of mass suggests that the equilibrium manifold sensitivity
analysis can be used as a design tool to optimize the car performance.
Thus, a deep investigation of these and other parameters on the shape of the equilibrium manifold (and hence on the nonlinear system behavior) would be an interesting area of research.

\section{Nonlinear optimal control based trajectory exploration}
\label{sec:traj_explor}
In this section we describe the optimal control based strategies used to explore the
dynamics of the car vehicle and provide numerical computations showing their
effectiveness.

\subsection{Exploration strategy based on least-square optimization}
Complex dynamic interactions make the development of maneuvers highly
nontrivial. To this end, we use nonlinear least squares trajectory optimization
to explore system trajectories. That is, we consider the optimal control problem
\begin{equation} \label{eq:OCP}
\begin{split}
  \min &\;\; \frac{1}{2}\int_0^T \|\xrm(\tau)-\xrm_{d}(\tau)\|_Q^2 +
  \|\urm(\tau)-\urm_{d}(\tau)\|_R^2d\tau  
+ \frac{1}{2}\|\xrm(T)-\xrm_{d}(T)\|_{P_1}^2\\
  \subj &\;\; \dot \xrm (t) = f(\xrm (t),\urm (t)) \qquad \xrm(0) = \xrm_0,
\end{split}
\end{equation}
where $Q$, $R$ and $P_1$ are positive definite weighting matrices, for
$z\in\real^n$ and $W\in\real^{n\times n}$ $\|z\|_W^2 = z^T W z$,
%
%
and $(\xrm_d(\cdot), \urm_d(\cdot))$ is a desired curve.

We propose an optimal control based strategy to solve the optimal control
problem~\eqref{eq:OCP} and compute aggressive vehicle trajectories.
The strategy is based on the projection operator Newton method, \cite{JH:02},
see Appendix~\ref{APP:PO_appr}.
%
We want to stress that the projection operator Newton method, as any other a
descent method, guarantees the convergence to a local minimum of the optimal
control problem in~\eqref{eq:OCP}. Thus, a naive application of this method (or
any other available optimal control solver) may let the algorithm converge to
a (local minimum) trajectory that is too far from the desired curve and does not
contain useful information on the vehicle capabilities.
%
In order to deal with this issue, we develop an exploration strategy based on
the following features: (i) choose a desired (state-input) curve that well
describes the desired behavior of the vehicle, (ii) embed the original optimal
control problem into a class of problems parametrized by the desired curve, and
(iii) design a continuation strategy to morph the desired curve from an initial
nonaggressive curve up to the target one.
%
%

\subsubsection{Desired Curve design}
First, we describe how to choose the desired curve. The path and the
velocity profile to follow on that path, are usually driven by the
exploration objective. Thus, the positions $x_{d}(t)$ and
$y_{d}(t)$ 
and the velocity $v_{d}(t)$, with $t \in [0, T]$, of the desired curve are
assigned. For example, in the next sections we describe two maneuvers where
we want to understand the vehicle capabilities in following respectively a
chicane at ``maximum speed'' and a real testing track at constant speed.

How to choose the other portion of the desired curve (i.e. the remaining
states and the inputs) strongly affects the exploration process.
In order to choose this portion of the desired curve, we use a quasi
trajectory that, with some abuse of notation, we call \emph{quasi-static
  trajectory}.

Given $x_{d}(t)$, $y_{d}(t)$, $v_{d}(t)$, and the curvature
$\sigma_{d}(t)$, $t \in [0,T]$, for each $t \in [0, T]$, we impose the
equilibrium conditions~\eqref{eq:constr} for the desired velocity and path
curvature at time $t$. That is, posing $v_{qs}(t) = v_{d}(t)$ and
$\dot{\psi}_{qs}(t) = v_{d}(t) \sigma_{d}(t)$, we compute the
corresponding equilibrium value for the sideslip angle, $\beta_{qs}(t)$,
the yaw rate, $\dot{\psi}_{qs}(t)$, and the yaw angle,
$\psi_{qs}(t)$, together with the steer angle, $\delta_{qs}(t)$, and
the rear and front longitudinal slips, $\kappa_{r_{qs}}(t)$ and $\kappa_{f_{qs}}(t)$,  by solving the nonlinear
equations~\eqref{eq:eq_manifold}.
Thus, the quasi-static trajectory $(\xrm_{qs}(t), \urm_{qs}(t))$, $t \in [0,
T]$, is given by
\begin{align*}
  \xrm_{qs}(t) =& [x_{d}(t), \, y_{d}(t), \, \psi_{qs}(t), \, v_{d}(t), \, \beta_{qs}(t), \, v_{d}(t)\sigma_{d}(t)]^T \, ,\\
  \urm_{qs}(t) =& [\delta_{qs}(t), \, \kappa_{r_{qs}}(t), \kappa_{f_{qs}}(t)]^T \,.
\end{align*}

\begin{remark}
  \label{rmk:closeness_quasi-static}
  We stress that the quasi-static trajectory is not an LT-CAR trajectory since
  it does not satisfy the dynamics.
  However, experience shows that, 
  for low values of the (longitudinal and lateral) accelerations, the
  quasi-static trajectory is close to the trajectory manifold.  \oprocend
\end{remark}

The above considerations suggest that the quasi-static trajectory represents a
reasonable guess of the system trajectory on a desired track for a given
velocity profile. Thus, when only the desired position and velocity curves are
available, we set the desired curve as the quasi-static trajectory,
i.e. $\xi_{d} = (\xrm_{qs}(\cdot), \urm_{qs}(\cdot))$.
In doing this choice we remember that the positions and velocity profiles are
the ones we really want to track, whereas the other state profiles are just a
guess. Thus, we will weight the first much more than the latter.

\begin{remark}
  Since we are interested in exploring ``limit'' vehicle capabilities, most of
  the times, as it happens in real prototype tests, we will study aggressive
  maneuvers characterized by high levels of lateral acceleration. Thus, it can
  happen that a quasi-static trajectory can not be found (we are out of the
  equilibrium manifold).  If this is the case, we generate the desired curve by
  using the linear tires car model, LT$^2$-CAR, discussed in
  Section~\ref{sec:tires_gen_forces}, so that higher lateral accelerations can
  be achieved.
  In this way we can compute the quasi-static trajectory, and thus the desired
  curve, for more aggressive path and velocity profiles. \oprocend
\end{remark}


\subsubsection{Initial trajectory and optimal control embedding}
With the desired curve in hand we still have the issue of choosing the initial
trajectory to apply the projection operator Newton method.
To design the initial trajectory, we could choose an equilibrium trajectory
(e.g. a constant velocity on a straight line). However, such naive initial
trajectory could lead to a local minimum that is significantly far from the
desired behavior or cause a relatively high number of iterations.
From the considerations in Remark~\ref{rmk:closeness_quasi-static}, we know that
a quasi-static trajectory obtained by a velocity profile that is not ``too
aggressive'' is reasonably close to the trajectory manifold.

These observations motivate and inspire the development of an embedding and
continuation strategy.
We parametrize the optimal control problem in \eqref{eq:OCP} with respect
to the desired curve. Namely, we design a \emph{family of desired curves} that
continuously morph a quasi-static trajectory with a ``non-aggressive'' velocity
profile into the actual desired (quasi-static) curve.

\subsubsection{Continuation Update rule}
We start with a non-aggressive
desired curve, $\xi^1_d = (\xrm^1_d(\cdot), \urm^1_d(\cdot))$, and choose as initial
trajectory, $\xi^1_0 $, the projection of the desired curve, $\xi^1_0 =
\PP(\xi^1_d)$. That is, we implement equation \eqref{eq:proj_oper_def} with
$(\alpha(\cdot), \mu(\cdot)) = (\xrm^1_d(\cdot),
\urm^1_d(\cdot))$.
Then, we update the temporary desired curve, $\xi^i_d$, with the new curve in
the family, $\xi^{i+1}_d$, (characterized by a more aggressive velocity profile
on the same track) and use as initial trajectory for the new problem the optimal
trajectory at the previous step. The procedure ends when an optimal trajectory
is computed for the optimal control problem where the temporary desired curve
equals the actual one. Next, we give a pseudo code description of the
exploration strategy. We denote PO\_Newt$(\xi_i, \xi_{d})$ the local minimum
trajectory obtained by implementing the projection operator Newton method
presented in Appendix~\ref{APP:PO_appr} for a given desired curve $\xi_{d}$ and
initial trajectory $\xi_{i}$.
\vspace{-2ex}
\begin{center}
  \begin{minipage}{0.7\linewidth}
\begin{algorithm}[H]
\caption{Exploration strategy}
\label{alg:exploration_strategy}
\begin{algorithmic}
\REQUIRE desired path and velocity $x_{d}(\cdot)$, $y_{d}(\cdot)$ and $v_{d}(\cdot)$
\STATE compute: desired curve $\xi_{d} = (\xrm_{qs}(\cdot), \urm_{qs}(\cdot))$;\\
design: $\xi_{d}^i$, $i\in\{1,\ldots,n\}$, 
s.t. $\PP(\xi_{d}^1) \simeq \xi_{d}^1$ and $\xi_{d}^n = \xi_{d}$;\\
compute: initial trajectory $\xi_{0}^1 = \PP(\xi_{d}^1)$.

\FOR{$i=1, \ldots, n$}
\STATE compute: $\xi_{opt}^i$ = PO\_Newt($\xi_0^{i}$, $\xi_{d}^i$);
\STATE set: $\xi_0^{i+1}=\xi_{opt}^i$;
\ENDFOR
\ENSURE $\xi_{opt} = \xi_{opt}^n$.
\end{algorithmic}
\end{algorithm}
\end{minipage}
\end{center}


%
%
%
%
%
%
%
% ...

\subsection{Aggressive maneuver on a chicane and model validation}
As first computation scenario we perform an aggressive maneuver by using a
Computer Aided Engineering (CAE) tool for virtual prototyping to generate the
desired curve. CAE tools for virtual prototyping allow car designers to create a
full vehicle model and perform functional tests, without realizing a physical
prototype, with a very high level of reliability.
%
%
As CAE tool, we use Adams/Car developed by MSC.Software. Adams is one of the
most used multibody dynamics tools in the automotive industry.

%

The objective of this computation scenario is twofold: (i) we show the
effectiveness of the exploration strategy in finding an LT-CAR trajectory close
to the desired curve, and (ii) we validate the LT-CAR model by showing that the
desired curve, which is a trajectory of the full Adams model, is in fact
``almost'' a trajectory of the LT-CAR model.

The desired curve is obtained as follows. We set as desired path the chicane
depicted in Figure~\ref{fig:validation_model_adams}.  To obtain the desired
velocity profile, we set the initial velocity to $150$ km/h ($41.67$ m/s), and
invoke an Adams routine that generates a velocity profile to drive the vehicle
on the given path at maximum speed under a maximum acceleration
(\texttt{a$_{\texttt{max}}$}). The remaining desired state curves are obtained
by means of an Adams closed loop controller that drives the (Adams) vehicle on
the given path with the given velocity profile. The desired inputs are set to
zero since they do not have an immediate correspondence with the inputs of the
Adams vehicle. They are weighted lightly, thus giving the optimization the
necessary freedom to track the states.  With this desired trajectory in hands,
to ``run'' the exploration strategy, we need to define the initial trajectory
and the continuation update rule for the desired trajectory morphing.

The exploration strategy for this maneuver is as follows.
Initially, we limit
the maximum acceleration parameter to $50\%$ of the desired one
(\texttt{a$_{\texttt{max}}$}$_0$ $=$ $50\%$ \texttt{a$_{\texttt{max}}$}).
This
gives a trajectory that can be easily projected to the LT-CAR model to get a
suitable initial trajectory.
Then, we increase the vehicle capabilities of a $10\%$
acceleration step-size until the desired maximum acceleration is reached. For each
intermediate step, we set the Adams trajectory as temporary desired trajectory
and the optimal trajectory at the previous step as initial trajectory. A pseudo
code of the strategy is given in the following table.

\begin{center}
  \begin{minipage}{0.7\linewidth}
\begin{algorithm}[H]
\caption{Exploration strategy for the chicane maneuver}
\label{alg:exploration_strategy_chicane}
\begin{algorithmic}
\STATE Run: Adams/Car with \texttt{path} = ``chicane''\\
$~~~~~$ compute: velocity profile with \texttt{a$_{\texttt{max}}$}$_0$ = $50\%$ \texttt{a$_{\texttt{max}}$}\\
$~~~~~$ run: closed-loop driver to get $\xi_{d}^{50\%}$\\
Compute: initial trajectory $\xi_0^{50\%} = \PP(\xi_{d}^{50\%})$\\
\FOR{$i = 50, \ldots, 100$}
\STATE Run: Adams/Car with \texttt{path} = ``chicane'' \\
$~~~~$ compute: velocity profile with \texttt{a$_{\texttt{max}}$}$_i$ =
$i\%$ \texttt{a$_{\texttt{max}}$}\\
$~~~~$ run: closed-loop driver to get $\xi_{d}^{i\%}$\\
\STATE Compute: $\xi_{opt}^{i\%}$ = PO\_Newt($\xi_0^{i\%}$, $\xi_{d}^{i\%}$);
\STATE Set: $\xi_0^{(i+10)\%}=\xi_{opt}^{i\%}$;\\
\ENDFOR
\ENSURE $\xi_{opt} = \xi_{opt}^{100\%}$.
\end{algorithmic}
\end{algorithm}
\end{minipage}
\end{center}


%

\begin{figure*}[!ht]
  \begin{center}
    \subfloat[\small The chicane maneuver.] 
    {\includegraphics[width=56mm, height=35mm]{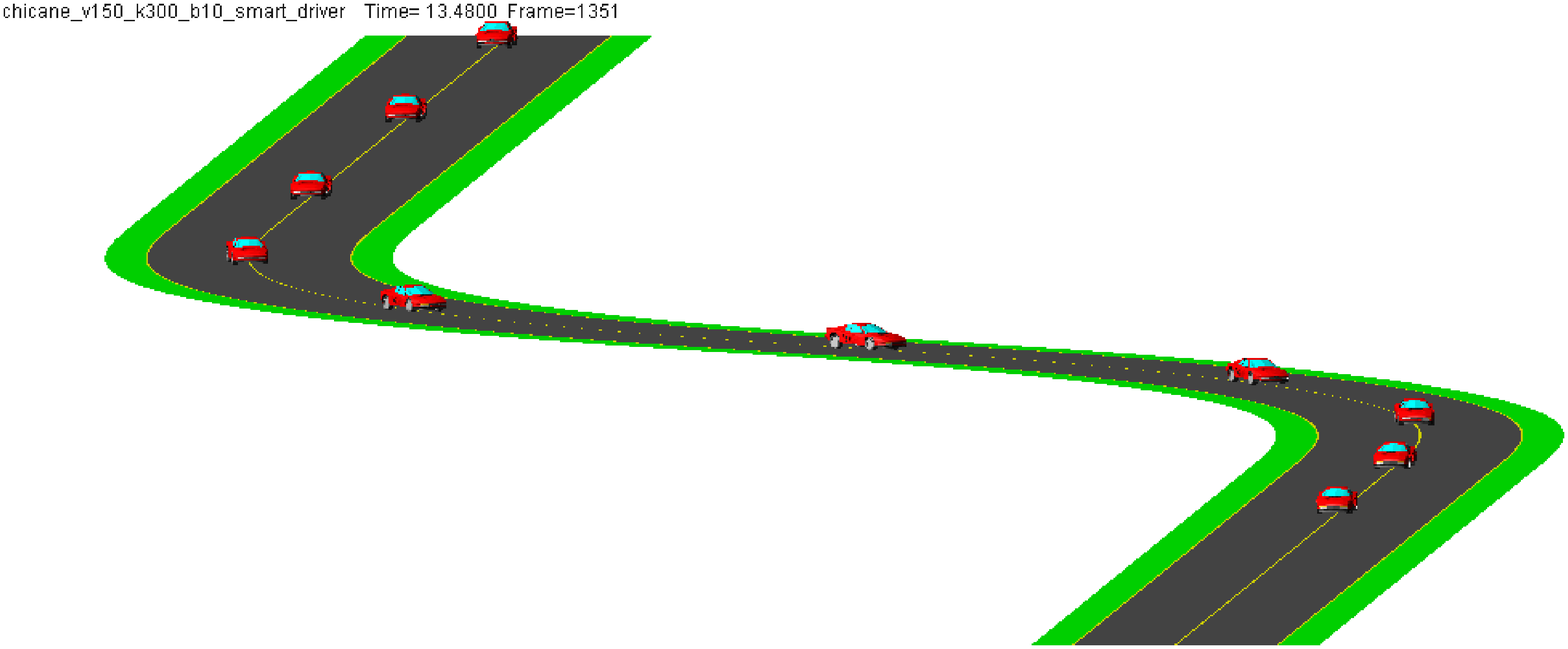} \label{fig:validation_model_adams}}
    \subfloat[\small Lateral acceleration $a_{lat}$.]
    {\includegraphics[scale=.375]{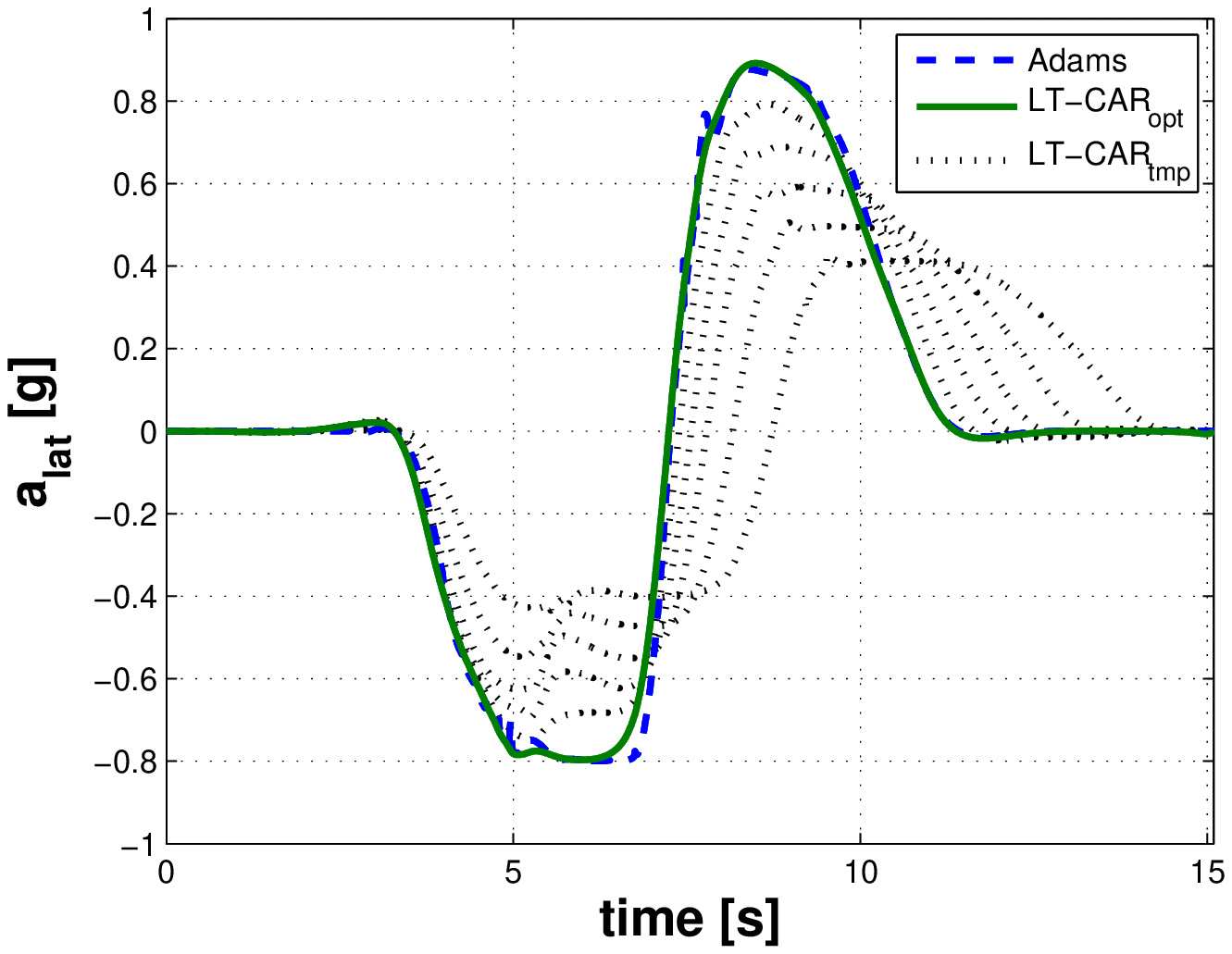} \label{fig:validation_model_alat}}
    \subfloat[\small Longitudinal velocity $v_x$.]
    {\includegraphics[scale=.375]{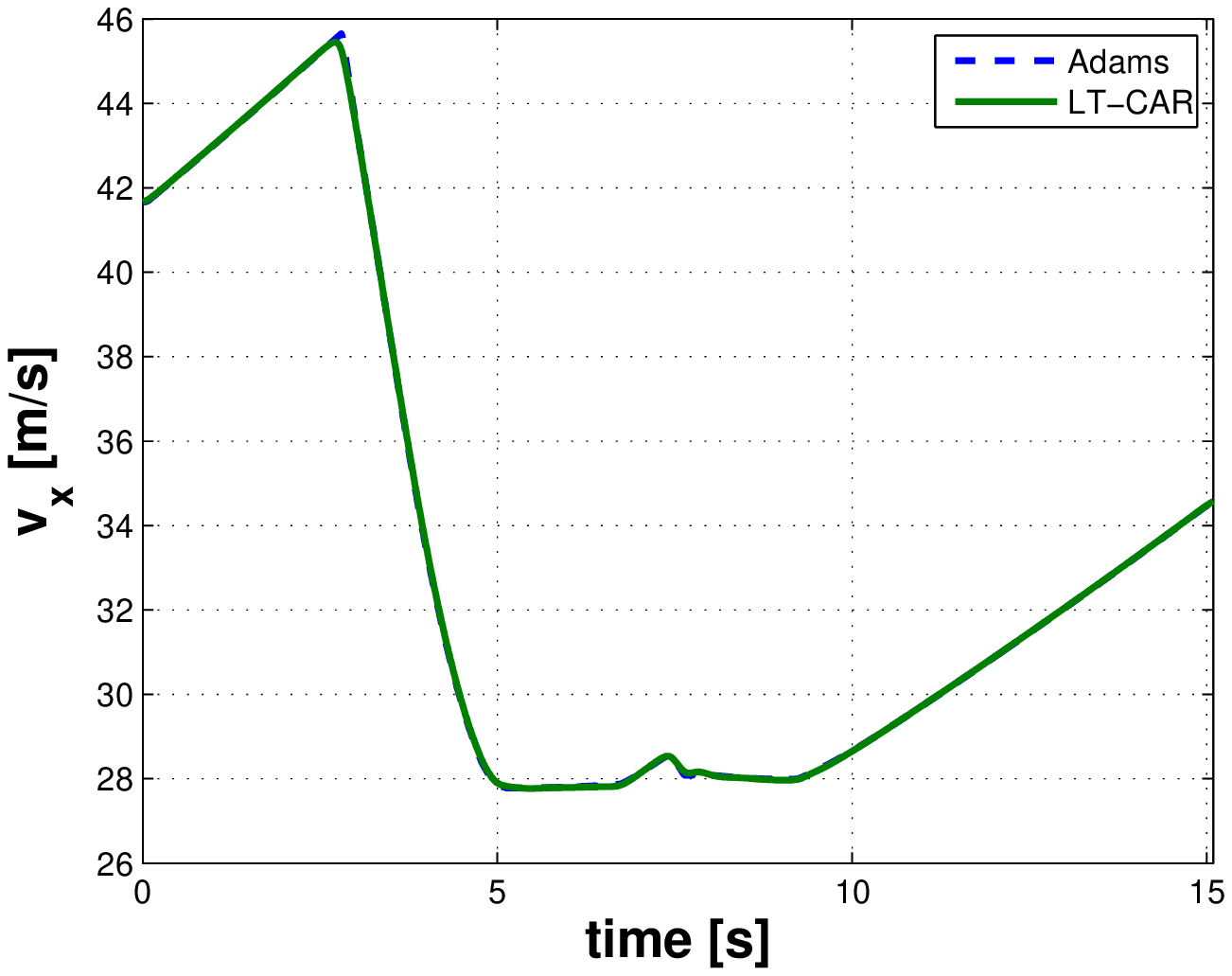} \label{fig:validation_model_vx}}

    \subfloat[\small Lateral velocity $v_y$.]
    {\includegraphics[scale=.375]{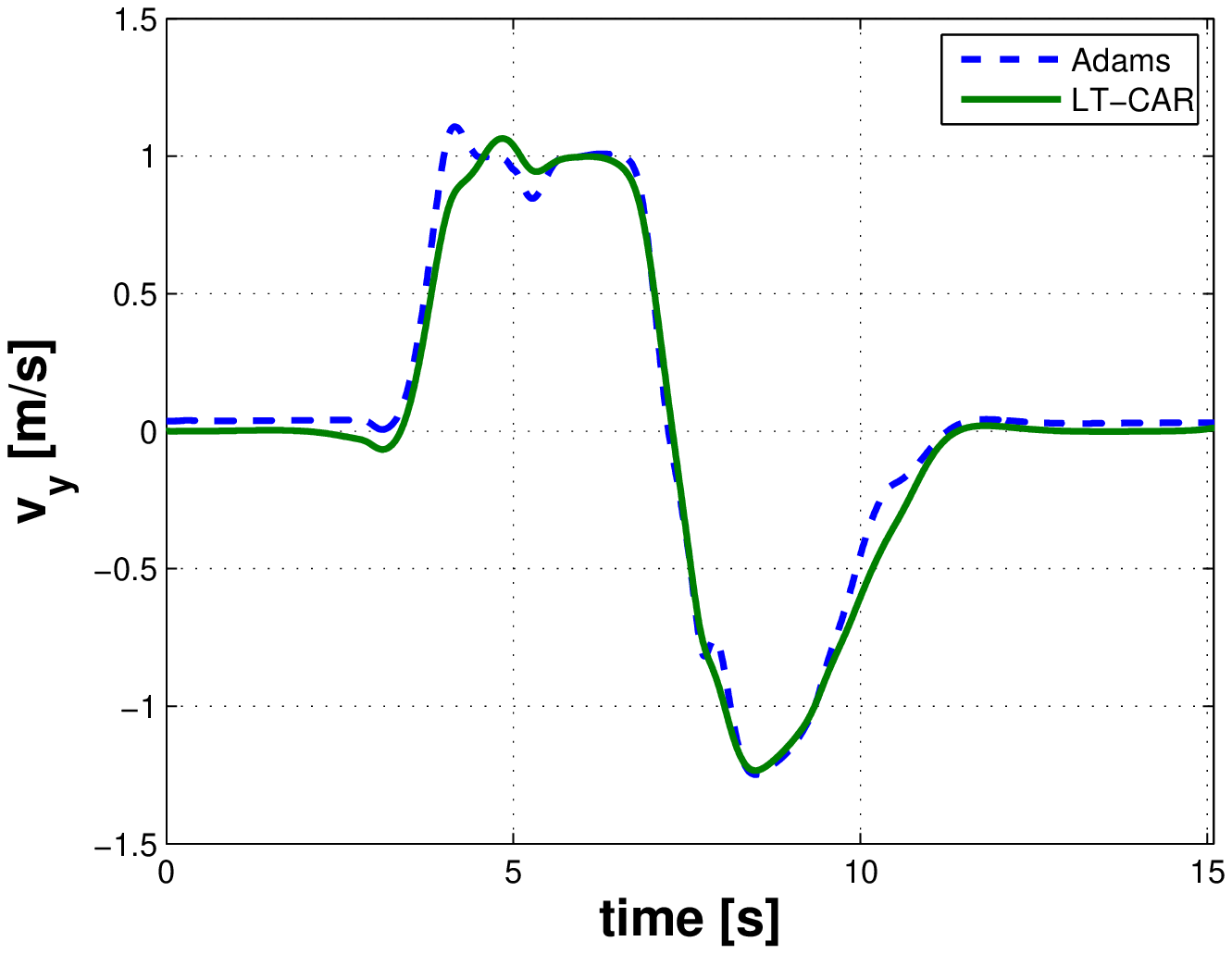} \label{fig:validation_model_vy}}
    \subfloat[\small Control inputs $\delta$(rad) and $\kappa$(Nodim).]
    {\includegraphics[scale=.375]{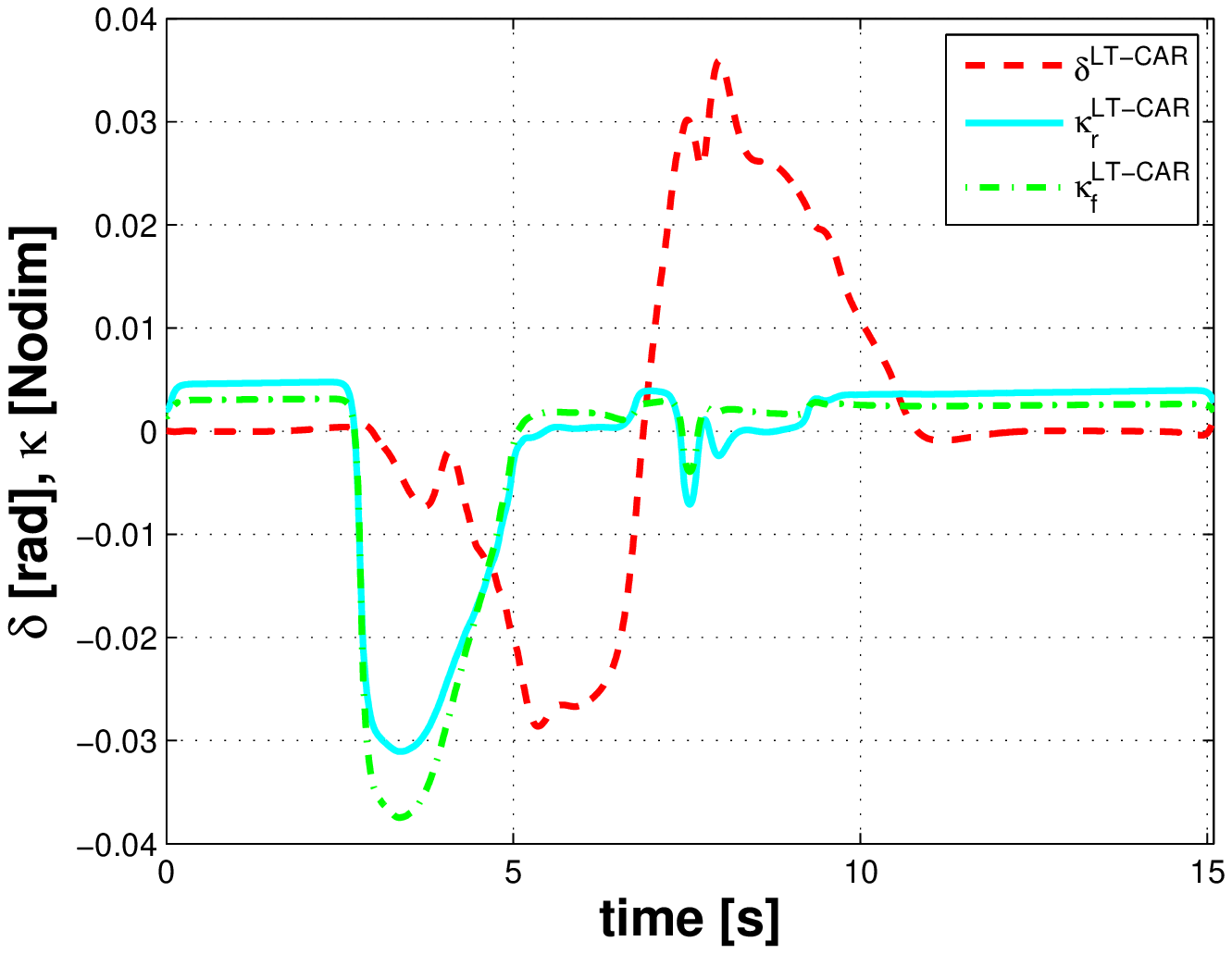} \label{fig:validation_model_inputs}}
%
    \subfloat[\small Wheel load $f_z$.]
    {\includegraphics[scale=.375]{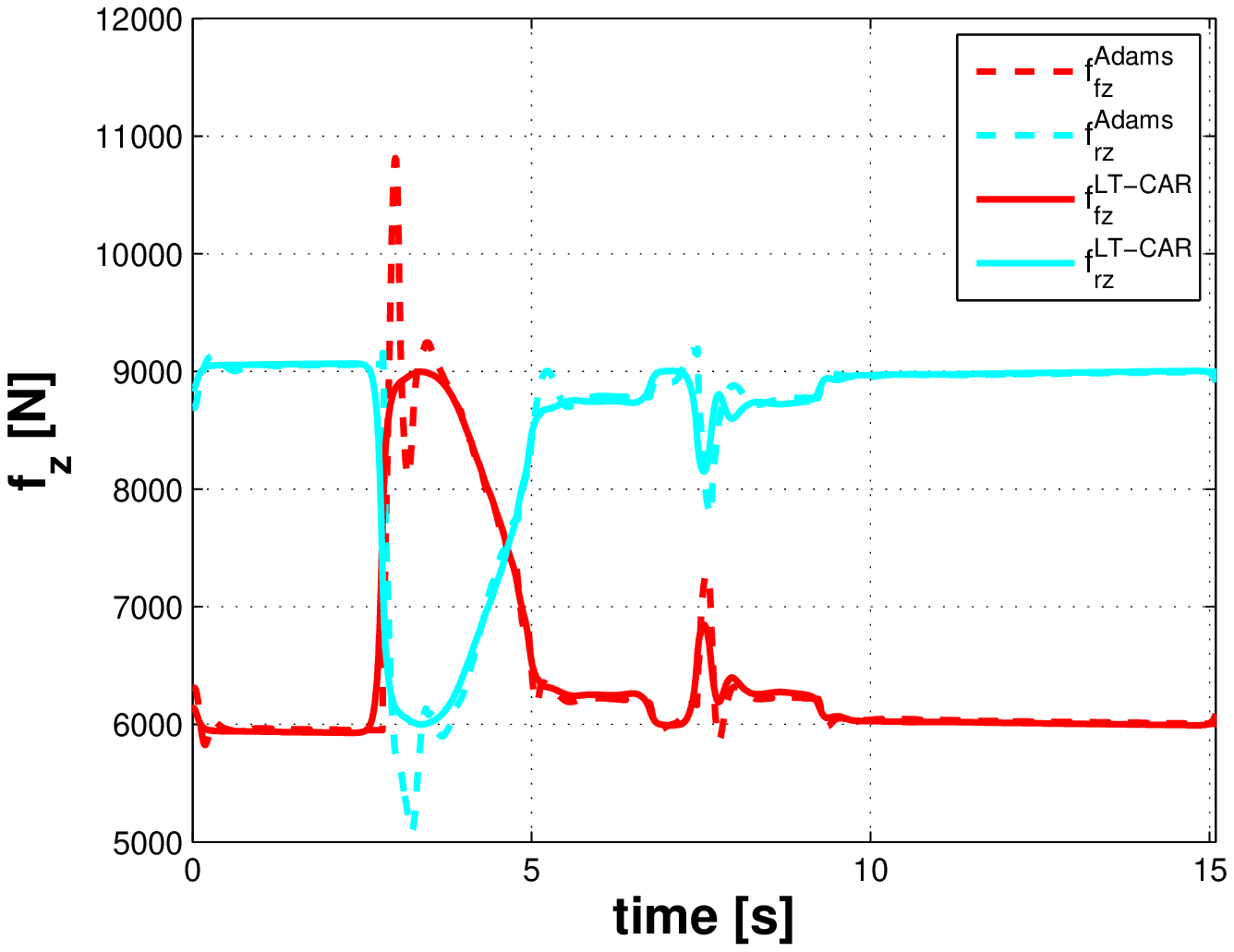} \label{fig:validation_model_fz}}
    \caption{\small Aggressive chicane maneuver.  
      The dash and solid lines are the Adams and the optimal LT-CAR
      trajectories, respectively (except for the input plots that are given only
      for the LT-CAR). Temporary optimal trajectories are in light dot lines.
      }
    \label{fig:validation_model}
  \end{center}
\vspace{-4ex}
\end{figure*}

%

In Figure~\ref{fig:validation_model} we show the main plots of the first
computation scenario. 
From the numerical computations we observed a fairly good position tracking. The
position error was less than $0.1$ m. In Figure~\ref{fig:validation_model_alat}
we show the lateral acceleration profile followed by the LT-CAR model versus the
Adams vehicle one. The light dot lines show the temporary optimal lateral
accelerations obtained during the continuation updates. In
Figure~\ref{fig:validation_model_vx} and Figure~\ref{fig:validation_model_vy} we
report respectively the longitudinal and lateral speed profiles. The maximum
error is less than $0.36$ m/s for the longitudinal speed and $0.07$ m/s for the
lateral one.
%
Comparing Figure~\ref{fig:validation_model_fz} with
Figure~\ref{fig:validation_model_vx}, we may notice the relationship between the
load transfer and the longitudinal acceleration (velocity slope).
The vehicle enters the first turn decreasing the speed (constant
negative slope) and the front load suddenly increases due to the load transfer
induced by the strong braking. After the first turn the velocity is slightly
increased (constant positive slope) as well as the load on the rear. Entering
the second turn, the vehicle reduces its speed again and then accelerates out
again.

It is worth noting in Figure~\ref{fig:validation_model_fz} how the LT-CAR
load transfer follows accurately the Adams vehicle load transfer except for a
high frequency oscillation (probably due to the Adams suspensions transient). We
stress the fact that there is an accurate prediction of the load transfer
although the LT-CAR has not a suspension model.

\subsection{Constant speed maneuver on a real testing track}
In this test the desired maneuver consists of following a real testing track at
constant speed%
\footnote{See http://www.nardotechnicalcenter.com/ for details on the track}.
In particular, we choose a desired speed that in the last turn gives a lateral acceleration
exceeding the tire limits.  For this reason we compute the desired curve as the
quasi-static trajectory of the Linear Tires LT-CAR model, (LT)$^2$-CAR, on the
desired path profile depicted in Figure~\ref{fig:ntc_path} with velocity
$v=30$ m/s.

The exploration strategy for this maneuver is as follows.
To morph to the desired curve, we start with a speed of $25$ m/s and increase the
velocity profile of $1$ m/s at each step.
%
%
For each speed value, we compute the desired curve as the quasi-static
trajectory of the (LT)$^2$-CAR model on the track. As mentioned before, for the
(LT)$^2$-CAR model we can find the quasi-static trajectory on a wider range of
lateral accelerations. The exploration strategy thus follows the usual steps. In
the following pseudo code we denote $\xi_{\text{LT$^2$-CAR}}^{v}$ the
quasi-static trajectory of LT$^2$-CAR obtained on the given path at constant
velocity $v$.


\begin{center}
  \begin{minipage}{0.7\linewidth}
\begin{algorithm}[H]
\caption{Exploration strategy for the constant speed maneuver}
\label{alg:exploration_strategy_constspeed}
\begin{algorithmic}
\REQUIRE desired path $x_{d}(\cdot)$, $y_{d}(\cdot)$ and $v_{d}(\cdot)\equiv30$m/s\\[1.1ex]
\STATE compute: desired curve $\xi_{d} = \xi_{\text{LT$^2$-CAR}}^{30\text{m/s}}$;\\[1.1ex]
  compute: initial trajectory $\xi_{0}^{25} = \PP(\xi_{\text{LT$^2$-CAR}}^{25\text{m/s}})$.\\

\FOR{$v=25, \ldots, 30$ m/s}
\STATE set: $\xi_{d}^{v}=\xi_{\text{LT$^2$-CAR}}^{v}$;
\STATE compute: $\xi_{opt}^v$ = PO\_Newt($\xi_0^{v}$, $\xi_{d}^v$);
\STATE set: $\xi_0^{v+1}=\xi_{opt}^v$;
\ENDFOR
\ENSURE $\xi_{opt} = \xi_{opt}^{30}$.
\end{algorithmic}
\end{algorithm}
\end{minipage}
\end{center}

\begin{figure*}[!ht]
  \begin{center}
    \subfloat[\small Path $x$-$y$.]
    {\includegraphics[scale=.256]{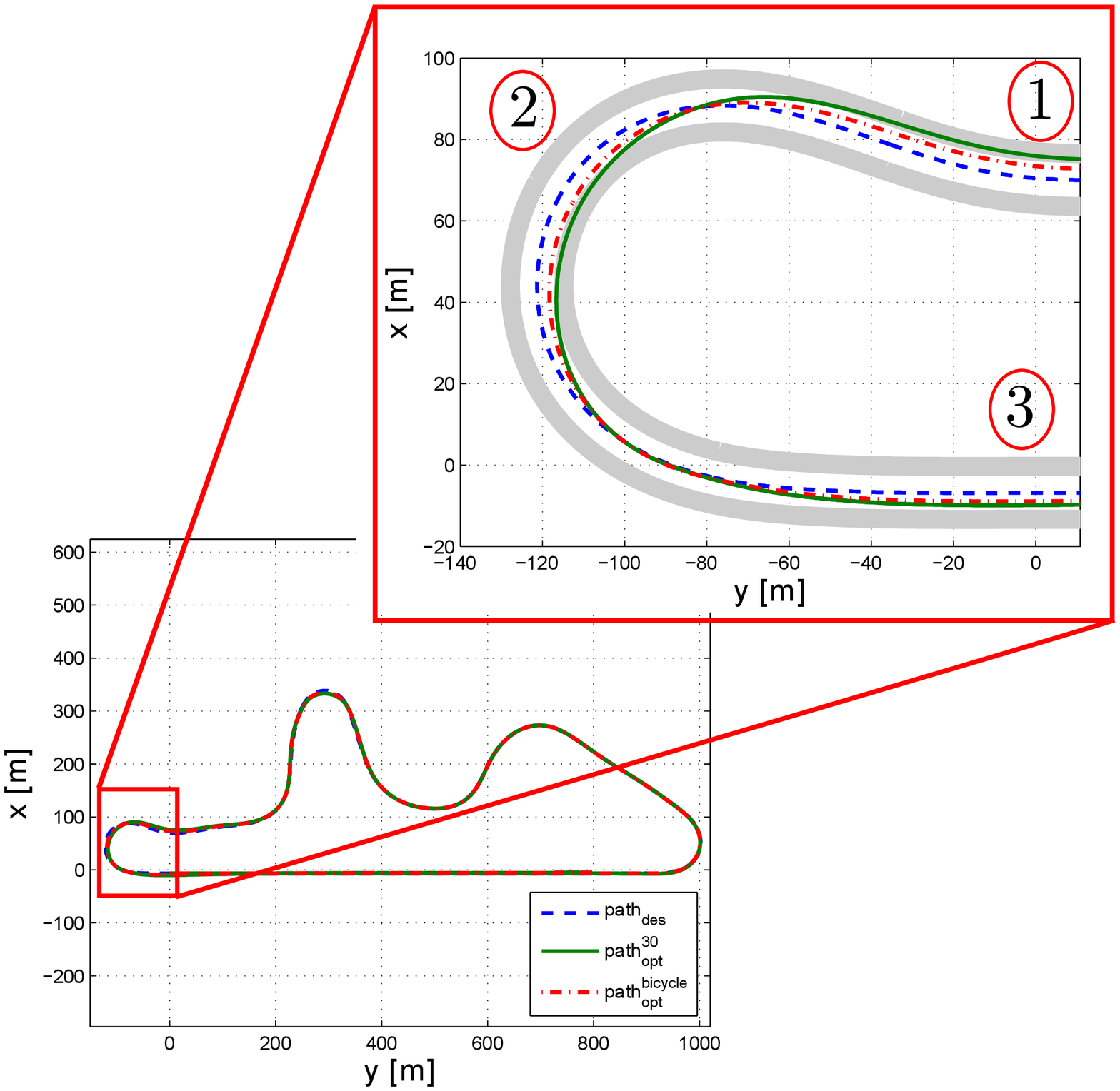} \label{fig:ntc_path}}
    \subfloat[\small Lateral acceleration $a_{lat}$.]
    {\includegraphics[scale=.375]{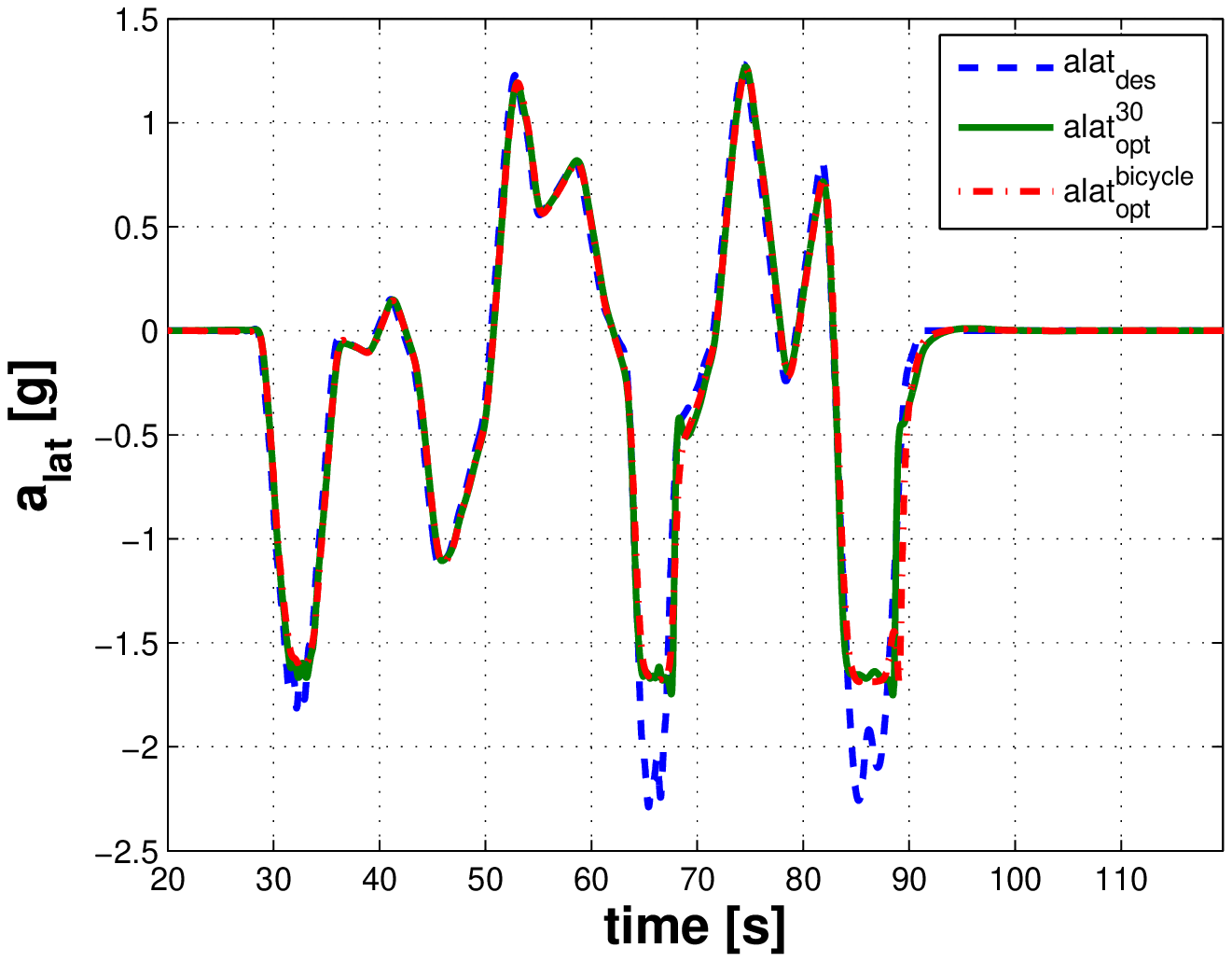} \label{fig:ntc_alat}}
    \subfloat[\small Velocity $v$.]
    {\includegraphics[scale=.375]{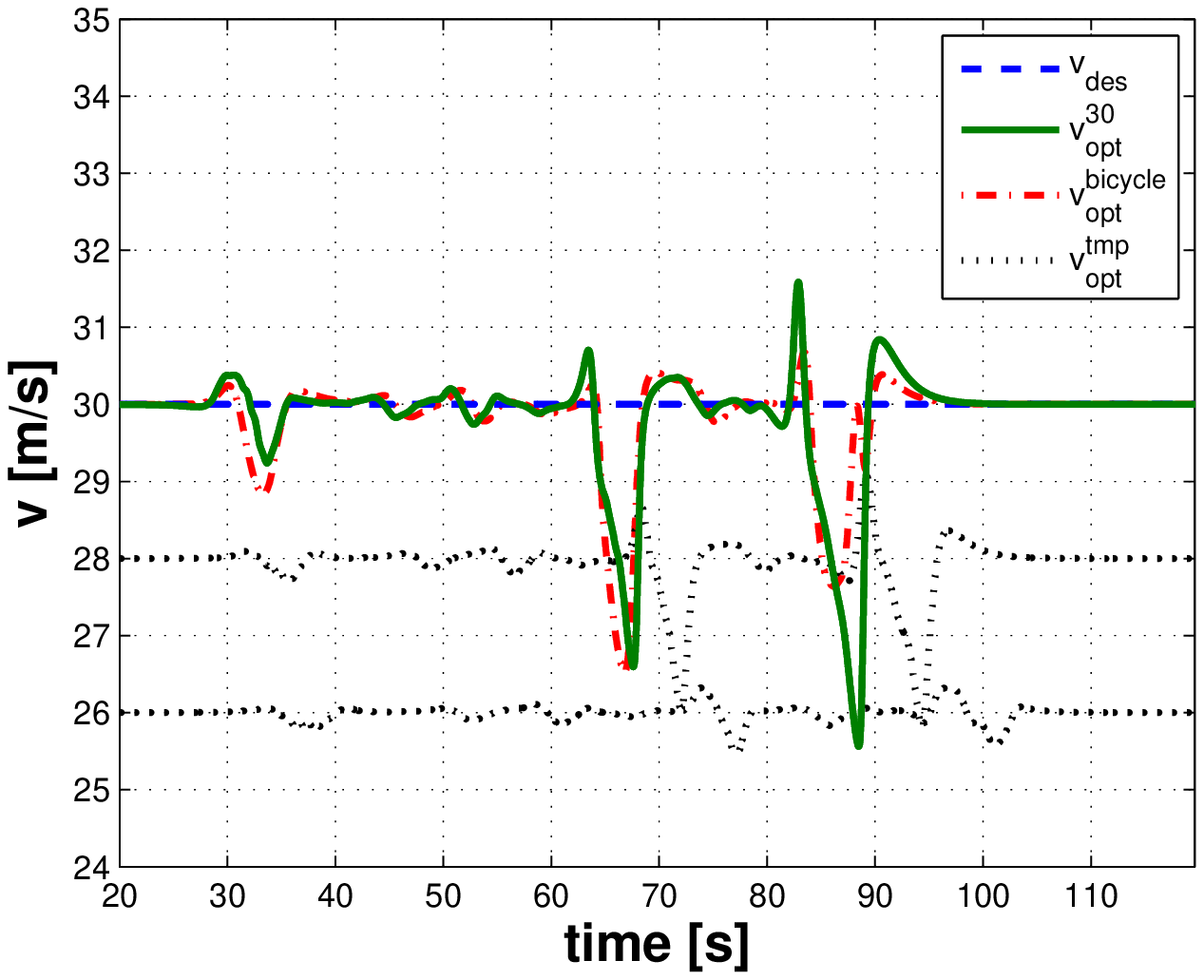} \label{fig:ntc_v}}

    \subfloat[\small Sideslip $\beta$.]
    {\includegraphics[scale=.375]{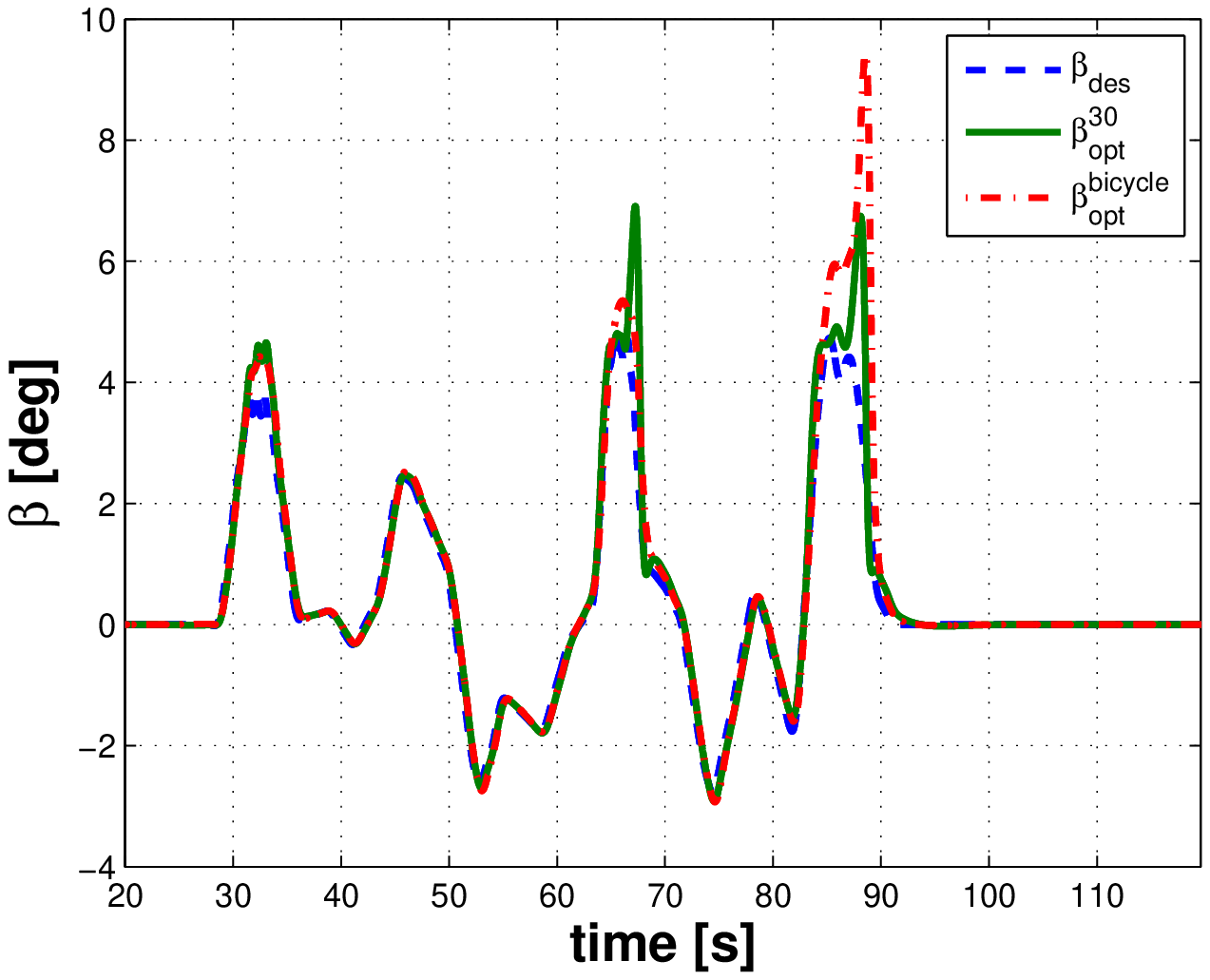} \label{fig:ntc_beta}}
    \subfloat[\small Steer angle $\delta$.]
    {\includegraphics[scale=.375]{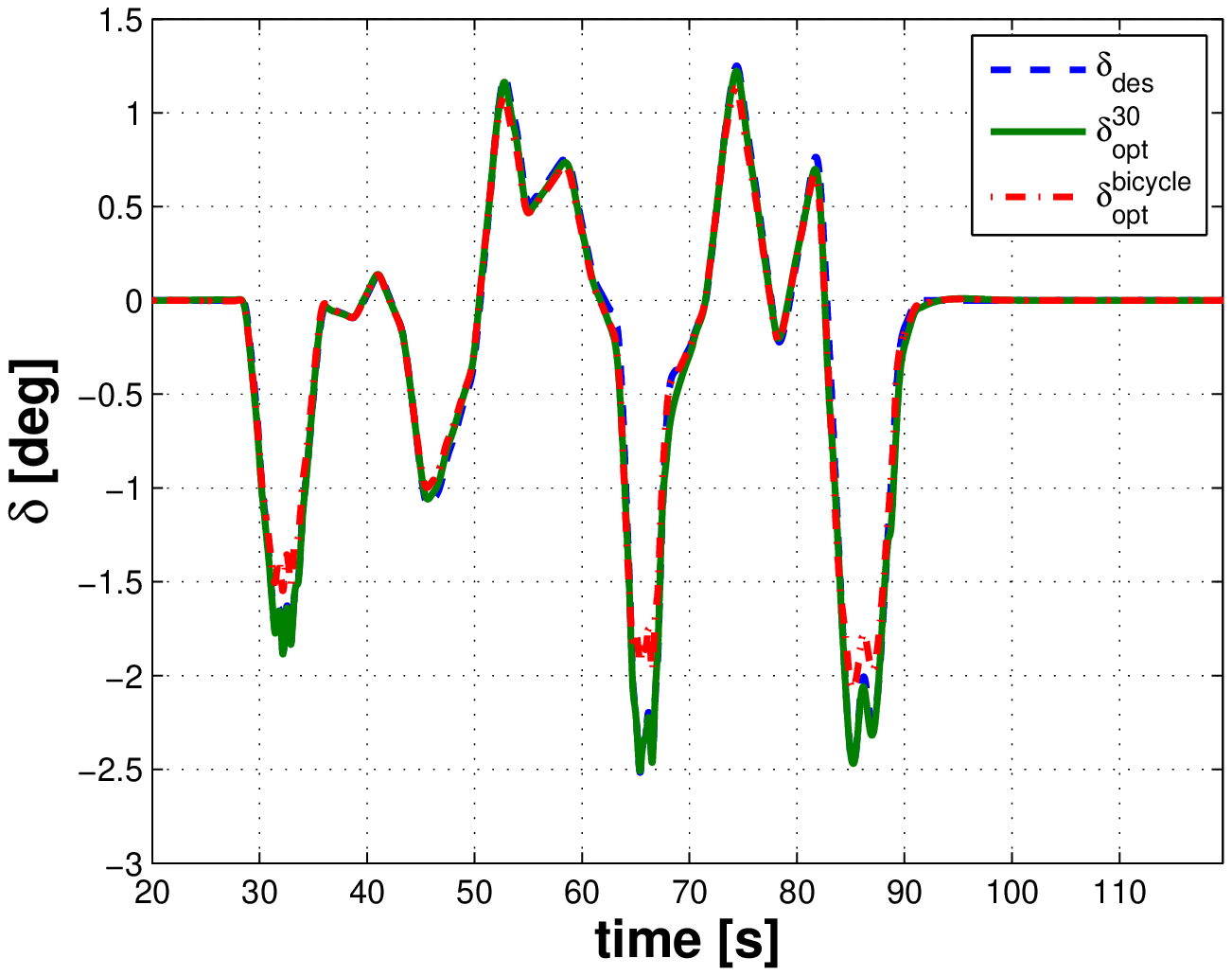} \label{fig:ntc_delta}}
    \subfloat[\small Longitudinal slip $\kappa$.]
    {\includegraphics[scale=.375]{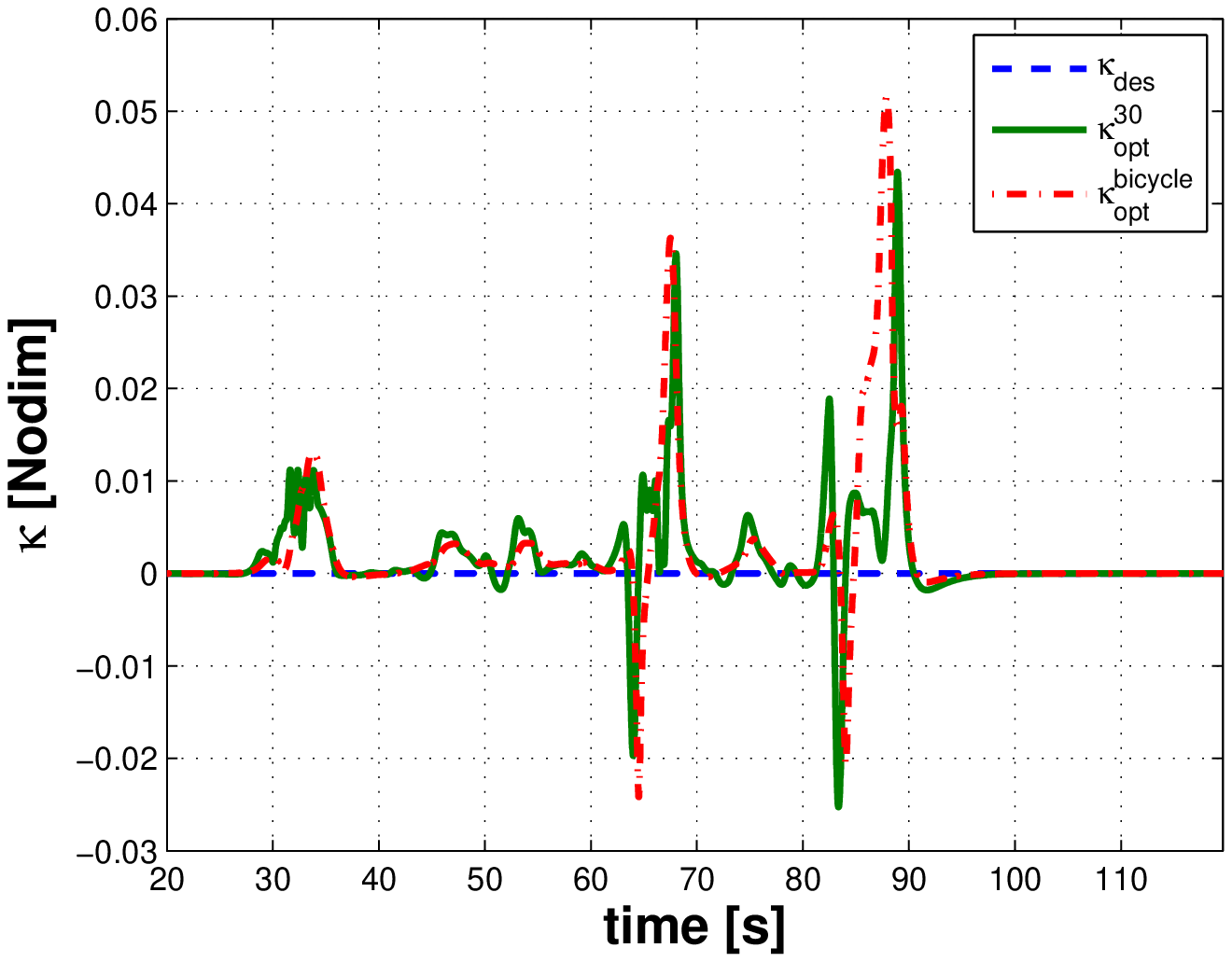} \label{fig:ntc_kappa}}

%

    \caption{\small Constant speed ($30m/s$) maneuver on a real
      testing track. The dash, solid and
      dash-dot lines are the desired curve, the optimal LT-CAR
      and the optimal bicycle model trajectories, respectively. Temporary optimal
      trajectories (for $v=26,28$ m/s)
      are in light dot lines.
    }
    \label{fig:ntc}
  \end{center}
\vspace{-4ex}
\end{figure*}

In Figures~\ref{fig:ntc} the optimal trajectory of the LT-CAR model (solid
green) is compared with the desired curve (dash blue) and with the optimal
trajectory of the bicycle model (dash-dot red).
We choose a desired speed ($30$ m/s) that in the last turn gives a lateral
acceleration exceeding the tire limits. 
%
The comparison with the bicycle model confirms the importance of including the
load transfer.
%
%
Indeed, the behavior of the two models is significantly different in the braking
and acceleration regions, Figure~\ref{fig:ntc_kappa}.
In particular, in the last turn, the vehicle decelerates in order to
satisfy the maximum acceleration limit. Notice that both the two models achieve
this limit, Figure~\ref{fig:ntc_alat}.  However, the LT-CAR has to anticipate
the breaking point and increase the value of the longitudinal slip with respect
to the bicycle model, Figure~\ref{fig:ntc_kappa}. This behavior is due to the
load transfer: in the LT-CAR the weight shifts to the front axle, thus reducing
the traction capability. The behavior is reversed in the acceleration region.
Notice that, in order to achieve the maximum lateral acceleration, the LT-CAR
requires a lower sideslip angle, but a higher steer angle,
Figures~\ref{fig:ntc_beta}~and~\ref{fig:ntc_delta}.
Next, we comment on an interesting phenomenon happening in the last turn.  In
the first straight portion (highlighted with ``$1$'' in
Figure~\ref{fig:ntc_path}), the vehicle moves to the right of
the track to reduce the path curvature when entering the turn.
%
%
In order to generate the required lateral forces 
in the turn (portion ``$2$'') 
the tires have a high sideslip angle, Figure~\ref{fig:ntc_beta}.
When the car starts to exit the turn (portion ``$3$''), 
%
the lateral forces on the tires decrease, so that the longitudinal slip can increase,
Figure~\ref{fig:ntc_kappa}, 
to regain the desired constant speed.

\section{Conclusions}
In this paper we studied the problem of modeling and exploring the dynamics of a
single-track rigid car model that takes into account tire models and load
transfer.
Starting from the bicycle model, we introduced the load transfer phenomenon by
explicitly imposing the holonomic constraints for the contact with the ground.
The resulting model shows many of the interesting dynamic effects of a real car.
For this rigid car model we characterized the equilibrium manifold on the entire
range of operation of the tires and analyzed how it changes with respect to
suitable parameters.
Finally, we provided a strategy, based on nonlinear optimal control techniques
and continuation methods, to explore the trajectories of the car model.
%
Specifically, the proposed exploration strategy provides an effective approach
for exploring the limits of the vehicle. The strategy was used, e.g., to find
trajectories in which the lateral acceleration limit of the vehicle is reached,
without applying constrained optimal control methods.
We provided numerical computations showing the effectiveness of the exploration
strategy on an aggressive maneuver and a real testing track.


%

\appendix


\subsection{Car model parameters}
\label{APP:params}
The tire equations introduced in Section~\ref{sec:tires_gen_forces} are based on
the formulation in \cite{HBP:02}. The pure longitudinal and lateral slips are given by
\[
\begin{split}
f_{x0}(\kappa)&=d_x\sin{\{c_x\arctan{[ b_x\kappa-e_x(b_x\kappa-\arctan{ b_x\kappa
    }) ]}\}},\\
f_{y0}(\beta)&=d_y\sin{\{c_y\arctan{[b_y\beta-e_y(b_y\beta-\arctan{b_y\beta})]}\}}
\end{split}
\]
and the loss functions for combined slips by
\begin{equation*}
  \begin{split}
    g_{x\beta}(\kappa,\beta) &=
    \cos{\biggl[c_{x\beta}\arctan{\Bigl(\beta\frac{r_{bx1}}{1+r^{2}_{bx2}\kappa^2}\Bigr)}\biggr]},\\
    g_{yk}(\kappa,\beta) &=
    \cos{\biggl[c_{yk}\arctan{\Bigl(\kappa\frac{r_{by1}}{1+r^{2}_{by2}\beta^2}\Bigr)}\biggr]}.
  \end{split}
\end{equation*}
The parameters are based on the ones given in \cite{GG:00}. The tire parameters
are determined by nonlinear curve-fitting routines. \\[1.2ex]
%
\begin{minipage}{.50\textwidth}
\center\small{Sports car}
\[
\footnotesize{
  \begin{array}{llllll}
     ~         &   ~~\mbox{rear}      &   ~~\mbox{front}     &      ~        &   ~~\mbox{rear}      &   ~~\mbox{front}    \\
    d_x      &   \phantom{-}1.688   &   \phantom{-}1.688   &     d_y     &   \phantom{-}1.688     &   \phantom{-}1.688     \\[1ex]
    c_x        &   \phantom{-}1.65    &   \phantom{-}1.65    &     c_y       &   \phantom{-}1.79     &   \phantom{-}1.79   \\
    b_x        &   \phantom{-}8.22  &   \phantom{-}8.22  &     b_y       &   \phantom{-}8.822  &   \phantom{-}12.848  \\
    e_x        &   \phantom{-}-10.0  &   \phantom{-}-10.0  &     e_y       &            - 2.02   &            -1.206   \\[1ex]
    c_{x\beta} &   \phantom{-}1.1231  &  \phantom{-}1.1231   &   c_{y\kappa} &   \phantom{-}1.0533  & \phantom{-}1.0533  \\
    r_{bx1}    &   \phantom{-}13.476  &  \phantom{-}13.476   &    r_{by1}    &   \phantom{-}7.7856  &  \phantom{-}7.7856\\
    r_{bx2}    &   \phantom{-}11.354  &  \phantom{-}11.354   &    r_{by2}    &   \phantom{-}8.1697  &  \phantom{-}8.1697\\
  \end{array}
  }
\]
\footnotesize{
\[
  a = 1.421 [m]
  \quad
  b = 1.029 [m]
  \quad
  h = 0.42 [m]
\]
\[
    m = 1480 [kg] \quad
 I_b =
  \left[
  \begin{array}{ccc}
    590 &   0      & -50 \\
      0      & 1730 &   0      \\
    -50 &   0      & 1950
  \end{array}
  \right]
\]
}
\end{minipage}%
\begin{minipage}{.50\textwidth}
\center\small{Adams model} 
\[
\footnotesize{
  \begin{array}{llllll}
     ~         &   ~~\mbox{rear}      &   ~~\mbox{front}     &      ~        &   ~~\mbox{rear}      &   ~~\mbox{front}    \\
    d_x      &   \phantom{-}1.48   &   \phantom{-}1.48   &     d_y     &   \phantom{-}1.22     &   \phantom{-}1.22     \\[1ex]
    c_x        &   \phantom{-}1.37    &   \phantom{-}1.37    &     c_y       &   \phantom{-}1.25     &   \phantom{-}1.25   \\
    b_x        &   \phantom{-}18.22  &   \phantom{-}18.22  &     b_y       &   \phantom{-}17.8  &   \phantom{-}17.8  \\
    e_x        &            -  0.46  &            -  0.46  &     e_y       &   \phantom{-}0.02   &   \phantom{-}0.02   \\[1ex]
    c_{x\beta} &   \phantom{-}1.1231  &  \phantom{-}1.1231   &   c_{y\kappa} &   \phantom{-}1.0533  & \phantom{-}1.0533  \\
    r_{bx1}    &   \phantom{-}13.476  &  \phantom{-}13.476   &    r_{by1}    &   \phantom{-}7.7856  &  \phantom{-}7.7856\\
    r_{bx2}    &   \phantom{-}11.354  &  \phantom{-}11.354   &    r_{by2}    &   \phantom{-}8.1697  &  \phantom{-}8.1697\\
  \end{array}
}
\]
\footnotesize{
\[
  a = 1.48 [m]
  \quad
  b = 1.08 [m]
  \quad
  h = 0.43 [m]
\]
\[
    m = 1528.68 [kg] \quad
 I_b =
  \left[
  \begin{array}{ccc}
    583.39 &   0      & -1.91 \\
      0      & 6129.12 &   0      \\
    -1.91 &   0      & 6022.36
  \end{array}
  \right]
\]
}
\end{minipage}

\subsection{Proof of Proposition \ref{prop:car_model}}
\label{APP:proof_model}
 This appendix gives the main steps for the derivation of the constrained Lagrangian dynamics.

  To prove statement (i), we use Lagrange's equations \eqref{eq:lagrange}
  including all the coordinates (even the constrained ones) and plug the
  constraints directly into the equations of motion (rather than attempting to
  eliminate the constraints by an appropriate choice of coordinates).  The
  constraints are taken into account by adding the constraint forces into the
  equation of motion as additional forces which affect the motion of the
  system.
  Hence the constrained equations of motion can be written as
 \[
  \begin{split}
    M(q)\ddot{q}+C(q,\dot{q})+G(q) &= J^T_f(q)f-A^T(q)\lambda  \\
    A(q)\ddot{q}+\dot{A}(q)\dot{q} &= 0,
  \end{split}
  \]
  where $M$, $C$, $G$ and $A$ are the one introduced in \eqref{eq:xypsiztheta}
  and \eqref{eq:constr_matrix}.
  The constraints lead to $q_c(t)=\dot{q}_c(t)=\ddot{q}_c(t)=0$, $\forall t \in
  \real$, so that we have
 \begin{align*}
    & [M(q)\ddot{q}+C(q,\dot{q})+G(q)]|_{q_c=0}=[J^T_f(q)f-A^T(q)\lambda]|_{q_c=0}
  \end{align*}
  where
  \begin{align*}
    &M(q)|_{q_c=0} = \left[
        \begin{array}{c|c}
          M_1(q_r) & M_2(q_r)
        \end{array}
      \right]
    = \left[
        \begin{array}{ccc|cc}
          m & 0 & -mbs_{\psi} & 0 & -mhc_{\psi}\\
          0 & m & mbc_{\psi} & 0 & -mhs_{\psi}\\
          -mbs_{\psi} & mbc_{\psi} & mb^2+I_{zz} & 0 & 0\\
          0 & 0 & 0 & m & mb \\
          -mhc_{\psi} & -mhs_{\psi} & 0 & mb & I_{yy}+m(b^2+h^2)\\
        \end{array}
      \right] 
  \end{align*}

  \[
  A(q)|_{q_c=0} = \left[
      \begin{array}{ccccc}
        0 & 0 & 0 & 1 & -(a+b) \\
        0 & 0 & 0 & 1 &  0\\
      \end{array}
    \right] \, ,
  \]
  %
  and $C(q,\dot{q})|_{q_c=0}$, $G(q)|_{q_c=0}$,
  $J^T_f(q)f|_{q_c=0}$ are given by 
  \eqref{th:lagr_const_G}, and \eqref{th:lagr_const_U} respectively.
  We rewrite the equations of motion with respect to the \emph{extended variables}
  $[q_r, \lambda]^T$ as
  \begin{equation} [M_1(q_r) | A^T]
     \left[
        \begin{array}{cc}
          \ddot{q}_r \\
          \lambda \\
        \end{array}
      \right] \,
    +\CC(q_r,\dot{q}_r)+\GG(q_r)=
    \left[
      \begin{array}{c}
        \UU_1 \\ \hline
        0
      \end{array}
    \right].
    \label{th:laststep}
  \end{equation}
  Defining $\tilde{\MM} = [M_1(q_r) | A^T]$, the special
  structure \eqref{th:lagr_const_matrix} follows.

  To prove statement (ii), we compute the reduced Lagrangian
  $\mathcal{L}_{r}(q_{r}) = T(q_{r},\dot{q}_{r})-V(q_{r})$ and derive the
  Euler-Lagrange equations. Explicit calculations, shown in
  Appendix~\ref{APP:reduced_model}, lead to equation \eqref{eq:xypsi}. The
  expression of the constraint forces follows from the arguments to
  prove statement (i).

  Finally, to prove (iii), if the forces $f$ depend linearly on the reaction
  forces we have $f=F\lambda$, for a suitable $F$, then we can rewrite the
  generalized forces as
  \begin{align*}
    &\biggl[
      \begin{array}{c}
        \UU_1 \\ \hline
        0
      \end{array}
    \biggr] = J^T_f(q_r)|_{q_c=0} \left[
        \begin{array}{cc}
          \mu_{fx} & 0 \\
          \mu_{fy} & 0 \\
          0 & \mu_{rx} \\
          0 & \mu_{ry} \\
        \end{array}
      \right]
    \lambda
     = 
        \begin{bmatrix}
          c_{\psi}\mu_{fx}-s_{\psi}\mu_{fy} & c_{\psi}\mu_{rx}-s_{\psi}\mu_{ry} \\
          s_{\psi}\mu_{fx}+c_{\psi}\mu_{fy} & s_{\psi}\mu_{rx}+c_{\psi}\mu_{ry} \\
          (a+b)\mu_{fy} & 0 \\ \hline
          0 & 0 \\
          0 & 0 \\
        \end{bmatrix}
      \lambda 
     := \left[
        \begin{array}{c}
          \MM_{12} \\ \hline
          0 \\
        \end{array}
      \right]\lambda,
  \end{align*}
  so that equation \eqref{th:laststep} becomes
  \begin{equation*}
      \tilde{\MM}(q_r)
      \left[
        \begin{array}{c}
          \ddot{q}_r \\
          \lambda \\
        \end{array}
      \right]
      -
      \left[
        \begin{array}{c|c}
          0 & \MM_{12}(q_r,\mu) \\
          \hline
          0 & 0
        \end{array}
      \right]
      \left[
        \begin{array}{c}
          \ddot{q}_r \\
          \lambda \\
        \end{array}
      \right]
      +\CC+\GG = 0
  \end{equation*}
from which equation \eqref{th:dynamic_eq} follows directly.


\subsection{Reduced order model without load transfer (bicycle model)}
\label{APP:reduced_model}
The vector $q_r = [x,y,\psi]^T$ provides a valid set of
generalized coordinates for dynamics calculations.  So, the equations of motion
for a Single-track rigid car with generalized coordinates $q_r=[x,y,\psi]^T$ are
given by
\begin{equation*}
  \MM_{11}(q_r)\ddot{q_r}+\CC_1(q_r,\dot{q}_r)+\GG_1(q_r)=\UU_1
\end{equation*}
where the mass matrix, the Coriolis and gravity vectors are
\[
\MM_{11}(q_r) = \left[
  \begin{array}{ccc}
    m & 0 & -mbs_{\psi} \\
    0 & m & mbc_{\psi} \\
    -mbs_{\psi} & mbc_{\psi} & (I_{zz}+mb^2) \\
  \end{array}
\right], \,
\CC_{1}(q_r,\dot{q}_r) = \left[
  \begin{array}{c}
    -mbc_{\psi}\dot{\psi}^2 \\
    -mbs_{\psi}\dot{\psi}^2 \\
    0 \\
  \end{array}
\right], \, 
\GG_1(q_r) = \left[
  \begin{array}{c}
    0 \\
    0 \\
    0 \\
  \end{array}
\right]
\]
and the vector of generalized forces is
%
\[
  \UU_1  = J_f^T(\psi)f = \left[
    \begin{array}{cccccc}
      c_\psi   & -s_\psi  & c_\psi   &  -s_\psi \\
      s_\psi   & c_\psi   & s_\psi   &   c_\psi \\
      0  &    (a+b) &       0  &        0  \\
    \end{array}
  \right] \left[
    \begin{array}{c}
      f_{fx}\\
      f_{fy}\\
      f_{rx}\\
      f_{ry}
    \end{array}
  \right].
\]


\subsection{Projection Operator Newton method}
\label{APP:PO_appr}
We recall the optimal control tools, namely the Projection Operator-based Newton
method, used to explore the trajectory manifold of the car vehicle, see
\cite{JH:02} and \cite{JH-DGM:98}. We are interested in solving optimal control
problems of the form
\[
 \begin{split}
   \min &\;\; h(\xi; \xi_d) \!:= \!\frac{1}{2}\!\!\int_0^T
   \!\!\!\!\|x(\tau)\!-\!x_{d}(\tau)\|_Q^2 +
   \|u(\tau)\!-\!u_{d}(\tau)\|_R^2d\tau 
   + \frac{1}{2}\|x(T)-x_{d}(T)\|_{P_1}^2\\
   \!\text{subj. \!to} &\;\; \dot x = f(x,u) \qquad x(0) = x_0,
 \end{split}
\]
with $\xi = (x(\cdot),u(\cdot))$ and $\xi_d = (x_d(\cdot), u_d(\cdot))$.
Denoting $\TT$ the manifold of bounded trajectories $(x(\cdot),u(\cdot))$ on
$[0,T]$, the optimization problem can be written as
\begin{equation}
  \min_{\xi\in\TT} h(\xi;\xi_d).
  \label{eq:constrOCP}
\end{equation}

%
The Projection Operator Newton method is based on a trajectory tracking approach, defining a projection operator that maps
a state-control curve (e.g., a desired curve) onto the trajectory
manifold. Specifically, the time varying-trajectory tracking control law
\begin{equation}
  \begin{split}
    \dot{x}(t) =& f(x(t), u(t)), \qquad x(0) = x_0,\\
    u(t) =& \mu(t) + K(t)(\alpha(t) - x(t))
  \end{split}
\label{eq:proj_oper_def}
\end{equation}
defines the projection operator
\[
\PP : \xi = (\alpha(\cdot), \mu(\cdot)) \mapsto \eta = (x(\cdot), u(\cdot)),
\]
mapping the curve $\xi$ to the trajectory $\eta$.

Using the projection operator to locally parametrize the trajectory manifold, we
may convert the constrained optimization problem \eqref{eq:constrOCP} into one
of minimizing the unconstrained functional $g(\xi;\xi_d) = h(P(\xi);\xi_d)$
using, for example, a Newton descent method as described below.
A geometric representation of the projection operator is shown in Figure~\ref{fig:proj_oper}.
\begin{figure}[ht!]
 \begin{center}
\vspace{-2ex}   
\includegraphics[scale=.25]{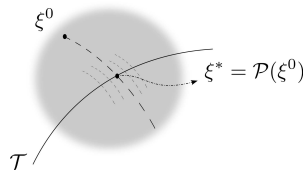}
\vspace{-2ex}
   \caption{\small Geometric representation of the trajectory manifold: every
     point of $\TT$ is a trajectory of the system. The projection of the curve $\xi^0 = (\alpha(\cdot),
     \mu(\cdot))$ on $\TT$ through $\PP$ is the trajectory $\xi^* = (\xrm(\cdot), \urm(\cdot))$.}
   \label{fig:proj_oper}
 \end{center}
\vspace{-4ex}
\end{figure}
%
Minimization of the trajectory functional is accomplished by iterating over the
algorithm shown in the table, where $\xi_i$ indicates the current trajectory iterate,
$\xi_0$ an initial trajectory, and $\zeta \mapsto Dg(\xi_i;\xi_d) \cdot \zeta$
and $\zeta \mapsto D^2 g(\xi_i;\xi_d)(\zeta,\zeta)$ are respectively the first
and second Fr\'echet differentials of the functional $g(\xi) = h(\PP(\xi);\xi_d)$ at
$\xi_i$.

%


 \noindent{\bf Projection operator Newton method (PO\_Newt)}\\
 Given initial trajectory $\xi_0 \in \TT$\\
 {\bf For} $i = 0, 1, 2 ...$
 \begin{itemize}
 \item[] design $K$ defining $\PP$ about $\xi_i$
 \item[] search for descent direction
   $$\zeta_i = \text{arg} \min_{\zeta\in T_{\xi_i} \TT} Dg(\xi_i;\xi_d) \cdot \zeta + \frac{1}{2} D^2 g(\xi_i;\xi_d)(\zeta,
   \zeta)$$
 \item[] step size $\gamma_i = \arg \min_{\gamma \in (0,1]} g(\xi + \gamma
   \zeta_i)$;
 \item[] project $\xi_{i+1}={\PP}(\xi_i + \gamma_i \zeta_i)$.
 \end{itemize} {\bf end}


The algorithm has the structure of a standard Newton method for the minimization
of an unconstrained function. The key points are the design of $K$ defining the
projection operator and the computation of the derivatives of $g$ to ``search
for descent direction''. It is worth noting that these steps involve the
solution of well known linear quadratic optimal control problems
\cite{JH:02}.





\end{document}